\documentclass[12pt]{amsart}
\usepackage{ amsmath, amsthm, amsfonts, amssymb, color}
\usepackage{mathrsfs}
\usepackage{amsfonts, amsmath}
\usepackage{amsmath,amstext,amsthm,amssymb,amsxtra}
\usepackage{txfonts} 
\usepackage[colorlinks, citecolor=blue,
hypertexnames=false]{hyperref}
\allowdisplaybreaks
\usepackage{pgf,tikz}
\usepackage{enumitem}

\newcommand\R{\mathbb{R}}

\newtheorem{theorem}{Theorem}[section]
\newtheorem{proposition}[theorem]{Proposition}

\newtheorem{lemma}[theorem]{Lemma}
\newtheorem{definition}[theorem]{Definition}

\newtheorem{remark}[theorem]{Remark}

\title[ Hodge Projections and Bounded Harmonic Functions on Manifolds with Ends]{On the Interplay Between Hodge Projections and Bounded Harmonic Functions on Manifolds with Ends}
\begin{document}

\author{Dangyang He and  Adam Sikora}
\address{
Dangyang He, School of Mathematical and Physical Sciences, Macquarie University, NSW 2109, Australia}
\email{dangyang.he@hdr.mq.edu.au}
\address{
Adam Sikora, School of Mathematical and Physical Sciences, Macquarie University, NSW 2109, Australia}
\email{adam.sikora@mq.edu.au}

\begin{abstract}
We investigate the $L^p$-boundedness of the Hodge projection in the setting of manifolds with ends. We examine its relationship to the Riesz transform and the space of bounded harmonic functions. In particular, we explore how the 
$L^p$-boundedness of the Hodge projection is connected to the structure of
$L^2$ harmonic one-forms and, subsequently, to the space of bounded harmonic functions.
\end{abstract}

\maketitle

\section{Introduction}

Our primary objective is to determine the range of exponents \(p\) for which the Hodge 
projection is bounded on \(L^p\) spaces over manifolds with ends. The Hodge projection and 
the associated Hodge decomposition form a classical and important topic in harmonic and 
geometric analysis. Foundational contributions to this field were made by Hodge and de~Rham 
(see \cite{DeRham, Hodge}). For historical context and subsequent developments, we refer the 
reader to \cite{Carron1, hode}. For significant aspects of Hodge theory leading in directions 
not directly related to our study, see \cite{ABP,Simpson}. For additional background on 
related issues of Hodge theory, we refer to \cite{Carron1, Li2,Li1}.

Our focus lies on the analytic aspects of the theory, particularly those connected to studies 
of the Riesz transform. In this context, our interest is motivated by the relationship between 
the \(L^p\)-boundedness of the Hodge projection and the equivalence (or lack thereof) between 
the Riesz and reverse Riesz transforms, see \eqref{rt} and  \eqref{rrt} below, as investigated by Auscher and Coulhon in \cite{AC}.

Recall that the Laplace--Beltrami operator can be defined via the closure of the quadratic form
\[
\langle \Delta f, g \rangle 
= \int_M \Delta f(x)\, g(x)\, dx 
= \int_M \langle df, dg \rangle_g,
\]
for all \(f,g \in C_c^\infty(M)\), where \(d\) denotes the exterior derivative. From this 
definition it follows that
\begin{align}\label{rt2}
  \|df\|_{2} = \|\Delta^{1/2} f\|_{2}.
\end{align}

In \cite{Str}, Strichartz posed the problem of determining the range of \(L^p(M)\) spaces 
for which the Riesz transform
\[
R = d\, \Delta^{-1/2}
\]
is bounded. That is, for which \(p\) the operator norm
\[
\|d \Delta^{-1/2}\|_{p \to p} = C < \infty
\]
is finite, or equivalently for which \(p\) the inequality
\begin{align}\label{rt}
  \|df\|_{p} \leq C \|\Delta^{1/2} f\|_{p}
\end{align}
holds. The identity \eqref{rt2} shows immediately that \(R\) is bounded on \(L^2(M)\).

Finally, the gradient \(\nabla f\) of a function is defined by the relation
\[
Xf = \langle \nabla f, X \rangle
\]
for all smooth vector fields \(X\).

 On a Riemannian manifold \( M \), the gradient \( \nabla f \) can be identified with \( df \in T^*M \) via the canonical identification of \( TM \) and \( T^*M \). Then the Riesz transform can be also denoted by $ R = \nabla  \Delta^{-1/2} $. In the sequel we always identify the tangent space with the 1-forms family. 

\begin{itemize}
    \item A significant answer to Strichartz’s question was obtained by Coulhon and Duong in \cite{CoDu}. However, the aspects of the Riesz transform closer to our present study were obtained in works by Coulhon, Carron, Hassell, Nix, and the second-named author (see \cite{CCH, HNS, HS}). The reverse Riesz transform can be defined by the inequality
\end{itemize}
\begin{equation}\label{rrt}
    \|\Delta^{1/2} f\|_p \le C \|\nabla f\|_p.
\end{equation}
This inequality and its relation to the Hodge projection constitute the main rationale for the results investigated in \cite{AC}, which also form an important starting point for our study. The aspects of the reverse Riesz transform most closely connected to our work are discussed by the first-named author in \cite{He2}, and we explain these connections in Remark~\ref{klopot} below.
To discuss the Hodge projection notion in a more systematic way we consider again  the exterior derivative  \(d\)  and define \(\delta\) as its formal adjoint with respect to the \(L^2\)-inner product, so that 
\[
\int \langle d\alpha, \beta \rangle_g \, dx = \int \langle \alpha, \delta \beta \rangle_g \, dx,\quad \forall \alpha \in C_c^\infty(\Lambda^{l-1}(M)),\quad \beta \in C_c^\infty(\Lambda^l(M)),
\]
where $\langle \cdot, \cdot \rangle_g$ is the metric-induced inner product on forms.

Let $M$ be a smooth, complete, oriented Riemannian manifold (possibly with boundary, under suitable boundary conditions).  
On the Hilbert space of square-integrable differential $l$-forms $L^2(\Lambda^l(M))$, the Hodge decomposition theorem asserts that
\begin{align}\label{Hodge_decomposition}
    L^2(\Lambda^l(M)) \;=\; \overline{d C_c^\infty(\Lambda^{l-1}(M))} 
    \;\oplus\; \overline{\delta C_c^\infty(\Lambda^{l+1}(M))} 
    \;\oplus\; \mathcal{H}_{L^2}^l(M),
\end{align}
where the closure is taken in the $L^2$-topology, and 
\[
\mathcal{H}_{L^2}^l(M) \;=\; \ker d \cap \ker \delta
\]
denotes the space of $L^2$-harmonic $l$-forms, see \cite{Carron-noncompact, KK}. The Hodge–de Rham Laplacian acting on $l$-forms is defined by
\[
\Delta_l = d \delta + \delta d,
\]
while on $0$-forms we simply write $\Delta$ for the usual Laplace–Beltrami operator.
In terms of a quadratic form 
\[
\langle \Delta_l \eta , \zeta \rangle = \int_M  \langle \Delta_l \eta,  \zeta \rangle_g dx =\int_M \langle d\eta , d\zeta \rangle_g +
\int_M \langle \delta \eta , \delta \zeta \rangle_g
\]
for all \( \eta, \zeta \in C_c^\infty(\Lambda^{l}(M)) \).

Equivalently, as shown for instance in \cite{Carron-noncompact}, one has
\begin{align*}
    \mathcal{H}_{L^2}^l(M) 
    &= \{ \omega \in L^2(\Lambda^l(M)) : \Delta_l \omega = 0 \} \\
    &= \{ \omega \in L^2(\Lambda^l(M)) : d\omega = \delta \omega = 0 \}.
\end{align*}

The \emph{Hodge projections} arise naturally from \eqref{Hodge_decomposition}: 
\begin{itemize}
    \item The projection $d \Delta_{l-1}^{-1} \delta$ is the orthogonal projection onto the exact component,
    \[
        d \Delta_{l-1}^{-1} \delta: L^2(\Lambda^l(M)) \to \overline{d C_c^\infty(\Lambda^{l-1}(M))}.
    \]
    \item The projection $\delta \Delta_{l+1}^{-1} d$ is the orthogonal projection onto the coexact component,
    \[
        \delta \Delta_{l+1}^{-1} d: L^2(\Lambda^l(M)) \to \overline{\delta C_c^\infty(\Lambda^{l+1}(M))}.
    \]
    \item The projection $\mathcal{H}$ is the orthogonal projection onto the harmonic component,
    \[
        \mathcal{H}: L^2(\Lambda^l(M)) \to \mathcal{H}_{L^2}^l(M).
    \]
\end{itemize}
Thus, for every $\omega \in L^2(\Lambda^l(M))$,
\begin{align*}
    \omega \;=\; d \Delta_{l-1}^{-1} \delta \omega 
    \;+\; \delta \Delta_{l+1}^{-1} d \omega 
    \;+\; \mathcal{H}\omega. 
\end{align*}
In other terms, for every \(l\)-form \(\omega \in L^2(\Lambda^l(M))\), we have the decomposition
\[
\omega = \alpha + \beta + \mathcal{H}\omega,
\]
with \(\alpha\) exact, \(\beta\) coexact, and \(\mathcal{H}\omega\) the harmonic component.

\medskip
Next, observe that $\Delta_l$ is non-negative and self-adjoint. By the spectral theorem, there exists a projection-valued measure $E(\cdot)$ on $[0,\infty)$ such that
\[
    \Delta_l = \int_{[0,\infty)} \lambda \, dE(\lambda).
\]
The \emph{Hodge projection} $\mathcal{H}$ is then given by
\[
    \mathcal{H} \;=\; E(\{0\}),
\]
that is, the spectral projection of $\Delta_l$ corresponding to the eigenvalue $0$. Equivalently,
\[
    \mathcal{H} = \chi_{\{0\}}(\Delta_l),
\]
where $\chi_{\{0\}}$ denotes the indicator function of the singleton $\{0\}$.


One central topic related to our investigation is the characterization of the class of 
\(L^2\)-harmonic 1-forms (see, for example, \cite{Carron-noncompact} and \cite{GW}). 
In this manuscript, however, our focus lies more on the consequences and extended applications 
of understanding this class of \(L^2\)-harmonic 1-forms rather than on developing a precise 
description of it.

The notation \( L^2 \)-harmonic 1-forms seems to be very closely related to the space of bounded harmonic function. At least in the setting of manifolds with ends which we investigate here the connection can have format of equivalent description. Pivotal contribution to studies of bounded harmonic functions can be found in \cite{Li, LiWa}. For some interesting application of this concept, see \cite{DY} and \cite{presik}.

For 1-form family we can write  $L^2$-Hodge decomposition described above in the following form
\begin{align}
    L^2(\Lambda^1(M)) = \overline{d C_c^\infty(M)}^{L^2} \oplus \overline{\delta C_c^\infty(\Lambda^2(M))}^{L^2} \oplus \mathcal{H}_{L^2}^1(M),
\end{align}
where
\begin{align}
    \mathcal{H}_{L^2}^1(M) = \left\{ \alpha \in L^2(\Lambda^1(M)); d\alpha = \delta \alpha = 0   \right\}
\end{align}
refers to the space of $L^2$ harmonic 1-form on $M$.

In what follows, we denote the orthogonal projection onto exact part by $\mathcal{P} = d \Delta^{-1} \delta$, the projection onto co-exact part by $\delta \Delta_2^{-1}d$, and the projection onto harmonic part by $\mathcal{H}$. That is
\begin{align*}
    I = d \Delta^{-1} \delta +\mathcal{H} + \delta \Delta_2^{-1} d.
\end{align*}

To state our results, we first recall the notion of a connected sum of smooth manifolds in a precise form. For further details on this construction, see \cite{GS,GIS}, and for additional context and motivation, we refer to \cite{Davies_nongaussian}.

\begin{definition}\label{def:connected-sum}
Let $\mathcal{V}_1, \ldots, \mathcal{V}_l$ be complete, connected, non-compact Riemannian manifolds of the same dimension.  
A Riemannian manifold $\mathcal{V}$ is called a \emph{connected sum} of $\mathcal{V}_1, \ldots, \mathcal{V}_l$, written
\[
\mathcal{V} \;=\; \mathcal{V}_1 \# \mathcal{V}_2 \# \cdots \# \mathcal{V}_l,
\]
if there exists a compact subset $K \subset \mathcal{V}$ such that the complement $\mathcal{V} \setminus K$ is the disjoint union of connected open subsets $\mathcal{U}_i \subset \mathcal{V}$, $i=1,\ldots,l$, where each $\mathcal{U}_i$ is isometric to $\mathcal{V}_i \setminus K_i$ for some compact subset $K_i \subset \mathcal{V}_i$.

The subsets $\mathcal{U}_i$ are called the \emph{ends} of $\mathcal{V}$.
\end{definition}






In our investigation we focus on the following three classes of manifolds with ends.
The first class is
\begin{align}\label{M1}\tag{$\textrm{M}_1$}
    M \;=\; (\mathbb{R}^{n_1} \times M_1) \# \cdots \# (\mathbb{R}^{n_l} \times M_l),
\end{align}
where $l \geq 2$ and each $M_i$ is a smooth compact Riemannian manifold. We assume
\[
n_* := \min_i n_i \geq 3, 
\qquad \text{and} \qquad 
n_i + \dim(M_i) = N \quad \text{for all } i.
\]
The boundedness of the Riesz transform on manifolds of this type was studied in \cite{CCH,HS}.

The second class is
\begin{align}\label{M2}\tag{$\textrm{M}_2$}
    M \;=\; (\mathbb{R}^2 \times M_1) \# (\mathbb{R}^{n_2} \times M_2),
\end{align}
where $n_2 \geq 3$, and
\[
n_2 + \dim(M_2) \;=\; 2 + \dim(M_1) \;=\; N.
\]
Here $M_1$ and $M_2$ are smooth compact Riemannian manifolds. The boundedness of the Riesz transform in this case was investigated in \cite{HNS}.

The third class is the connected sum
\begin{align}\label{M3}\tag{$\textrm{M}_3$}
    M \;=\; \mathbb{R}^2 \# \mathbb{R}^2.
\end{align}

Proofs of the above result, together with those of the subsequent theorems, are presented in the following sections. In the Introduction we confine ourselves to stating the main claims.

\begin{theorem}\label{thm1}
Let $M$ be a manifold with ends defined by \eqref{M1}. Then
\begin{align}\label{eq_main}
    \|\mathcal{P}\omega\|_p \le C \|(I-\mathcal{H})\omega \|_p,\quad \forall\, 1<p<n_*,
\end{align}
for all $\omega \in C_c^\infty(\Lambda^1(M))$.

In addition, the following endpoint estimate holds
\begin{align}
    \|\mathcal{P}\omega\|_{n_*} \le C \|(I-\mathcal{H})\omega \|_{(n_*,1)},
\end{align}
for all $\omega \in C_c^\infty(\Lambda^1(M))$.

\end{theorem}

In the above result, $\|\cdot\|_{(p,q)}$ refers to the standard Lorentz seminorm. We mention that the endpoint estimate may not be sharp, but since it follows directly from our proof without extra effort, we include it here for completeness.

\begin{remark}
Note that, by duality and Theorem~\ref{thm1}, for any \(\omega,\nu\in C_c^\infty(\Lambda^1(M))\) we obtain the bilinear estimate
\[
\bigl|\langle \mathcal{P}\omega,\nu\rangle\bigr|
= \bigl|\langle \mathcal{P}(I-\mathcal{H})^{2}\omega,\nu\rangle\bigr|
= \bigl|\langle (I-\mathcal{H})\omega,\mathcal{P}(I-\mathcal{H})\nu\rangle\bigr|
\le C\,\|(I-\mathcal{H})\omega\|_{p}\,\|(I-\mathcal{H})\nu\|_{p'},
\]
for all \(n_*'<p<\infty\).
However, since \((I-\mathcal{H})C_c^\infty(\Lambda^1(M))\) is not, in general, dense in \(L^{p'}(\Lambda^1(M))\), this duality argument alone does not yield \eqref{eq_main} for \(p\ge n_*\).
\end{remark}

The following observation is necessary for describing the $L^p$ boundedness of the Hodge projection. We expect that this result can be significantly strengthened, but a more detailed analysis will be presented in a forthcoming project.

\begin{theorem}\label{thm0}
Suppose that 
$M = \mathbb{R}^{n} \# \mathbb{R}^{n}$ for $n\ge 3$ or that $M$ is defined by \eqref{M1} and that its fundamental group $\pi_1(M)$ is trivial. Then
\[
\mathcal{H}_{L^2}^1(M) = \operatorname{span} \{d h_i\},
\]
where $h_i$ is the harmonic function equal to the probability of escaping to infinity through the end $E_i$.
Moreover 
$$
\|\mathcal{H}\|_{p \to p}\sim \|I-\mathcal{H}\|_{p \to p} < \infty 
$$
if and only if $n_*'<p < n_*$.
Now consider the case \eqref{M3}, that is, assume $M = \mathbb{R}^2 \# \mathbb{R}^2$. Then
\[
\mathcal{H}_{L^2}^1(M) = \{0\}.
\]
Moreover, if $M$ satisfies condition \eqref{M2} and its fundamental group $\pi_1(M)$ is trivial, then again
\[
\mathcal{H}_{L^2}^1(M) = \{0\}.
\]
\end{theorem}

As a consequence, it follows that:

\begin{theorem}\label{thm2}
Let $M$ be a manifold defined by \eqref{M2} with $\pi_1(M) = \{0\}$ or \eqref{M3}. Then the Hodge projector $\mathcal{P}$ satisfies the bounds
\begin{align}
    \|\mathcal{P}\omega\|_p \le C \|\omega\|_p, \quad \forall\, 1<p<\infty,
\end{align}
for all $\omega \in C_c^\infty(\Lambda^1(M))$.
\end{theorem}

Theorem~\ref{thm2} was inspired by results obtained in \cite{HS1D}, where it was noticed that in radial model of $\R^2\#\R^2$ the Hodge projection is bounded for all $1 \le p \le \infty$ but 
the Riesz transform is only bounded for $p \le 2$.

\begin{remark}
In fact, Theorem~\ref{thm2} follows by a self-improvement argument. Indeed, one can adapt the ideas from the proof of Theorem~\ref{thm1}; see Sections~\ref{sec3} and \ref{sec4} below. 
\end{remark}

Suppose that for some \( 1 < p < \infty \) we have
\[
\| d \Delta^{-1/2} f \|_p \le \|f\|_p 
\]
for all \( f \in C_c^\infty(M) \). This implies that the definition of the Riesz transform \( d \Delta^{-1/2} \) extends to all functions \( f \in L^p(M) \), and the above inequality holds on this larger class.  

The class of functions for which the \emph{reverse} Riesz transform \eqref{rrt} is valid is more delicate, and understanding this class is of central importance for clarifying the relationship between the Riesz transform, the reverse Riesz transform, and the Hodge projection. We summarise the key logical steps below.

\begin{remark}\label{klopot}
From \cite{HS,He1}, the following is known about the Riesz transform on \( M \), defined by \eqref{M1}:
\begin{enumerate}[label=\arabic*)]
    \item \textbf{Boundedness range:} The Riesz transform is bounded on \( L^p \) if and only if \( 1 < p < n_* \), is of weak type \( (1,1) \), and is bounded from \( L^{n_*,1} \) to \( L^{n_*,1} \).

    \item \textbf{Connection to the Hodge projection:}  
    Using the identity
    \[
    d\Delta^{-1}\delta = \left(d\Delta^{-1/2}\right) \left(d\Delta^{-1/2}\right)^*,
    \]
    it follows that the Hodge projection \( \mathcal{P} \) is bounded on \( L^p \) for \( n_*' < p < n_* \). 

    \item \textbf{Necessity of the range:}  
    This condition on \( p \) is also necessary. By \cite{He2}, the reverse Riesz inequality
    \[
    \|\Delta^{1/2} f\|_p \le C \|d f\|_p
    \]
    holds for all \( 1 < p < \infty \). If \( \mathcal{P} \) were bounded for some \( p \le n_*' \), then from \cite{AC} one would deduce (see Lemma~\ref{le_remedy2} below for detail)
    \begin{align}\label{AC_lemma}
        \| \Delta^{-1/2} \delta \omega\|_{p} 
        &= \| \Delta^{1/2} \Delta^{-1} \delta \omega\|_p 
        \le C \| \mathcal{P} \omega \|_p 
        \le C \|\omega\|_p,
    \end{align}
    which would imply boundedness of the Riesz transform on \( L^{p'} \), contradicting \cite{HS}.

    \item \textbf{Additional conditions:}  
    The above implication requires that \( \Delta^{-1} \delta \omega \) is in the class of functions which satisfies
    the reverse Riesz transform \eqref{rrt} so it satisfies  certain decay conditions. This is essential, and in some cases (e.g., in dimension \( 2 \)) the implication fails.

    \item \textbf{Example in dimension two:}  
    We show that the Hodge projection on \( \mathbb{R}^2 \# \mathbb{R}^2 \) is  bounded for all \( 1 < p < \infty \). In \cite{He2}, the first-named author proves that the reverse Riesz transform is bounded for the same range, while \cite{HNS} shows that the Riesz transform is bounded only for \( 1 < p \le 2 \).
\end{enumerate}

\medskip
We emphasise that \eqref{AC_lemma} is more subtle than it appears: it fails on the manifold defined by \eqref{M3}, and \cite[Lemma~0.1]{AC} requires certain simple but essential assumptions (see Section~\ref{Sec7} for details).  Essentially the same calculation as in \eqref{AC_lemma} and in  \cite[Lemma~0.1]{AC}  is used also  in Shen, see \cite[(3.6) calculation]{shen} 
\end{remark}

\begin{remark}
A general criterion due to Shen \cite[Theorem~3.1]{shen} (see also \cite[Theorem~2.3]{AC}) may be used to establish the boundedness of $\mathcal{P}$ on $L^p$. This theorem states that if a manifold $M$ satisfies the doubling condition \eqref{D} together with an $L^q$–$L^2$ reverse Hölder inequality, then $\mathcal{P}$ is bounded on $L^p$ for all $2<p<q$. In particular, by \cite[Example~4]{CJKS}, the reverse Hölder inequality holds on $\mathbb{R}^n \# \mathbb{R}^n$ ($n \geq 3$) for every $1<q<n$. Consequently, Shen’s result implies the boundedness of $\mathcal{P}$ on $L^p$ whenever $n' < p < n$, where $n' = n/(n-1)$. This range agrees precisely with that obtained via the Riesz transform method (see Remark~\ref{klopot}~(2)). By contrast, for the non-doubling manifolds considered in this article, the applicability of Shen’s criterion remains unresolved.
\end{remark}

In this article, we show that the only obstruction to the $L^p$-boundedness of $\mathcal{P}$ for $1<p \leq n_*'$ arises from the harmonic component of the $1$-form. 

The terms responsible for the failure of boundedness are closely related to those encountered in the case of the Riesz transform, namely the regime $\lvert y \rvert \gg \lvert x \rvert$ in the kernel of $\Delta^{-1/2}(x,y)$. Recall the asymptotic expansion from \cite{CCH}: let $\overline M$ denote a compactification of $M = \mathbb{R}^n \# \mathbb{R}^n$ ($n \geq 3$), obtained by attaching two spheres at the ends. Then the Riesz potential satisfies
\begin{align}\label{asy}
    \Delta^{-1/2}(x,y) \sim \sum_{j = n-1}^\infty a_j(x)\,|y|^{-j},
\end{align}
as $y \to \partial \overline M$, with $x \in \overline M$. The leading coefficient $a_{n-1}(x)$ is a nonconstant bounded harmonic function on $\overline M$. By the maximum principle, $d a_{n-1}$ does not vanish identically, so the leading term in the $y$-variable decays only as $\lvert y\rvert^{1-n}$. Such decay pairs boundedly only with functions in $L^{\,n-\epsilon}$ ($\epsilon>0$). This explains why the Riesz transform on $M$ is bounded on $L^p$ precisely for $1<p<n$.

Next, we observe that
\[
\mathcal{P} \;=\; d\Delta^{-1/2}\,(d\Delta^{-1/2})^* \;=\; R R^*,
\]
where $R$ denotes the Riesz transform on $\overline M$. By \cite{CCH}, it follows that $\mathcal{P}$ is bounded on $L^p$ whenever $n'<p<n$.  

For $f,g \in C_c^\infty(\Lambda^1(\overline M)) \subset L^2(\Lambda^1(\overline M))$, consider the bilinear form
\[
\langle \mathcal{P}f,\,g\rangle 
= \langle (I-\mathcal{H})f,\, \mathcal{P}g\rangle
= \iint d\Delta^{-1/2}(x,y)\, R^* g(y)\,dy \;\; (I-\mathcal{H})f(x)\,dx.
\]
We focus on the contribution from the region $y \in \partial \overline{M}$ and $x \in \overline{M}$. Using the asymptotic expansion \eqref{asy}, the leading term yields
\[
\int \Bigg(\int d a_{n-1}(x)\,(I-\mathcal{H})f(x)\,dx\Bigg)\,|y|^{1-n}\,R^* g(y)\,dy.
\]


Suppose now that
\begin{align}\label{decom}
    (I-\mathcal{H})f \;=\; d\alpha \;+\; \delta\beta,
\end{align}
with $\alpha \in C_c^\infty(\overline M)$ and $\beta \in C_c^\infty(\Lambda^2(\overline M))$.  
Disregarding boundary terms, an integration by parts shows that the inner contribution vanishes, since 
\[
\delta(d a_{n-1}) = \Delta a_{n-1} = 0 
\quad\text{and}\quad 
d(d a_{n-1}) = 0.
\] 
Thus, when dualised against a $1$-form without a harmonic component, the problematic leading term is annihilated. By a standard duality argument, this suggests that the lower threshold for the $L^p$-boundedness of $\mathcal{P}$ (namely $p\le n'$) can in fact be relaxed.  

Indeed, our main results establish that in the present setting
\begin{align*}
    \big|\langle \mathcal{P}f,\,g\rangle\big| 
    \;\le\; C\, \|(I-\mathcal{H})f\|_p \,\|g\|_{p'}, 
    \qquad \forall\, 1<p<n.
\end{align*}
The upper threshold ($p=n$) persists, reflecting the boundedness of $R^*$, i.e.
\[
\|R^*\|_{p\to p}\le C 
\quad \text{if and only if} \quad 
n'<p<\infty.
\] 
Whether this upper bound is sharp, however, remains an open question.

This perspective may also be understood in terms of an interior harmonic annihilation mechanism.  
Indeed, owing to the symmetry of the kernel of the Riesz potential, the obstruction to the boundedness of $R^*$ arises in the regime $|x| \gg |y|$, where
\begin{align}\label{asy2}
    \Delta^{-1/2}(x,y) \;\sim\; \sum_{j = n-1}^\infty |x|^{-j}\, a_j(y).
\end{align}

By \eqref{asy2} and the decomposition \eqref{decom}, the leading term of $R^*(I-\mathcal{H})f(x)$ in the range $|x| \gg |y|$ expands as
\begin{align*}
 |x|^{1-n} \int a_{n-1}(y)\, \delta (I-\mathcal{H})f(y)\,dy 
   = |x|^{1-n} \int a_{n-1}(y)\, \Delta \alpha(y)\,dy 
   = 0,
\end{align*}
which eliminates the lower threshold for the $L^p$-boundedness of $R^*$, namely $p=n'$.  

However, since $\mathcal{P}=RR^*$ and $R$ appears on the left, the upper threshold in the estimate
\begin{align*}
    \|\mathcal{P}f\|_p \;\le\; C \,\|(I-\mathcal{H})f\|_p
\end{align*}
remains at $p=n$.  

We emphasise that we do not assert the optimality of this upper bound $p<n$.

\section{Preliminaries}

\subsection{Parametrix for resolvent at low energy} 
In this section, we recall the construction of the parametrix for the resolvent operator at low energy on various classes of manifolds with ends, following \cite{CCH,HS,HNS}.  
Let $\Delta_i$ denote the Laplace--Beltrami operator on $\mathbb{R}^{n_i} \times M_i$. For each end, fix a reference point $z_i^0 \in K_i$, which will serve as the origin.  

We begin by recalling two key lemmata, \ref{keylemma_HS} and \ref{keylemma_HNS}, as stated in \cite{HS,HNS}.  
Throughout the article, we assume $0<k_0 \leq \tfrac{1}{2}$ and use the notation
\[
|z| := \sup_{a \in K} d(z,a).
\]
In particular, $|z| \geq C$ for all $z \in M$ sufficiently far from the compact core.

\begin{lemma}\cite[Lemma~2.7]{HS}\label{keylemma_HS}
Let $M$ be defined by \eqref{M1} with $l\ge 2$ and $n_* = \min_i n_i \ge 3$. Let $v \in L^{\infty}(M)$ with compact support in $K$. Then, there is a function $u: M \times \mathbb{R}^+ \to  \mathbb{R}$ such that $(\Delta + k^2)u = v$ and on the $i$th end
\begin{gather*}\label{2.7e}
|u(z, k)| \le  C \|v\|_\infty \begin{cases}
    1, & z\in K,\\
    |z|^{2-n_i} e^{- ck d(z_i^0, z)}, & z \in  E_i.
\end{cases}
\\
|\nabla u(z, k)| \le  C \|v\|_\infty \begin{cases}
    1, & z\in K,\\
    |z|^{1-n_i} e^{- ck d(z_i^0, z)}, & z\in E_i,
\end{cases}
\end{gather*}
for all $0\le k\le k_0$, $1\le i\le l$. 

\end{lemma}

Next, we introduce the function
\begin{align*}
    \operatorname{ilg}(k) \;:=\; 
    \begin{cases}
        \dfrac{-1}{\log k}, & 0<k \leq \tfrac{1}{2}, \\[6pt]
        0, & k=0.
    \end{cases}
\end{align*}
With this notation, the key lemma of \cite{HNS} can be formulated as follows.

\begin{lemma}\cite[Lemma~2.14]{HNS}\label{keylemma_HNS}
Let $M$ be defined by \eqref{M2} with $n_2\ge 3$. Let $v \in C_c^{\infty}(M)$. Suppose $\phi$ is the solution to the equation $\Delta \phi = -v$. Then, for every integer $q$, there  exists an approximate solution $u(z,k)$ such that $u(z,0)=-\phi(z)$ and
\begin{align*}
    (\Delta + k^2)u - v = \chi_{K} O((\textrm{ilg}(k))^q |z|^{-\infty}) + O(k |z|^{-\infty}),
\end{align*}
where the notation $f = O(|z|^{-\infty})$ is shorthand for the statement $|f|\le C_\kappa |z|^{-\kappa}$ for any $\kappa \ge 0$. 

Moreover, one has estimates
\begin{gather*}
|u(z, k)| \le \begin{cases}
    C, & z\in K,\\
    C |z|^{2-n_2} e^{- ck |z|}, & z \in  E_2,\\
    C e^{-ck |z|}, & z \in E_1.
\end{cases}\\  
|\nabla u(z, k)| \le \begin{cases}
    C, & z\in K,\\
    C |z|^{1-n_2} e^{- ck |z|}, & z \in  E_2,\\
    C \left[|z|^{-2} + (\textrm{ilg}(k))|z|^{-1}\right] e^{-ck|z|}, & z \in E_1
\end{cases}
\end{gather*}
for all $0\le k\le k_0$.

In addition, 
\begin{align}\label{deri1}
|\nabla \left(\phi(z)+u(z,k)\right)| = O(\textrm{ilg}(k)|z|^{-1}) \chi_{E_1}(z) + O(k|z|^{-1}).    
\end{align}

\end{lemma}




\begin{remark}
We record the following observations:  
\begin{enumerate}[label=\arabic*)]
\item The function $\operatorname{ilg}(k)$ tends to zero as $k \to 0$, but more slowly than any polynomial $k^{\alpha}$ with $\alpha > 0$.  
\item Compared with the original formulation, the error estimate for $(\Delta + k^2)u - v$ is sharper, since the ``bad'' term $\operatorname{ilg}(k)$ is supported only on the compact set $K$.  
\item The estimate for $\nabla(\phi + u)$ in \eqref{deri1} follows directly from the original proof without additional work.  
\end{enumerate}
In the special case $M = \mathbb{R}^2 \# \mathbb{R}^2$, points 2) and 3) will be established rigorously in Lemma~\ref{keylemma_proof} below.
\end{remark}

Next, adapting the argument of \cite[Lemma~2.14]{HNS}, we extend the key lemma to the case $M=\mathbb{R}^2 \# \mathbb{R}^2$. The proof is deferred to Appendix~\ref{appendix}.

\begin{lemma}\label{keylemma}
Let $M=\mathbb{R}^2\# \mathbb{R}^2$.  
For $v \in C_c^\infty(K)$, let $\phi$ denote the solution to $\Delta \phi = -v$ given by \cite[Lemma~2.9]{HNS}. Then, for any integer $q$, there exists an approximate solution $u(z,k)$ to $(\Delta+k^2)u=v$ such that $u(z,0)=-\phi(z)$ and
\begin{align*}
    (\Delta+k^2)u - v \;=\; \chi_K\, O(\operatorname{ilg}(k)^q |z|^{-\infty}) + O(k |z|^{-\infty}),
\end{align*}
with the bounds, valid for all $0 \leq k \leq k_0$,
\begin{align*}
    |u(z,k)| &\le C \begin{cases}
        1, & z \in K, \\[4pt]
        e^{-ck|z|}, & z \in E_i,
    \end{cases} \\[6pt]
    |\nabla u(z,k)| &\le C \begin{cases}
        1, & z \in K, \\[4pt]
        \big(|z|^{-2} + |z|^{-1}\operatorname{ilg}(k)\big)\, e^{-ck|z|}, & z \in E_i.
    \end{cases}
\end{align*}
Moreover,
\begin{align}\label{deri2}
    |\nabla(\phi(z)+u(z,k))| \;=\; O(\operatorname{ilg}(k)\,|z|^{-1}).
\end{align}
\end{lemma}

\begin{proof}
See Appendix~\ref{keylemma_proof}.
\end{proof}

Let $\phi_i \in C^\infty(M)$ with $\operatorname{supp}(\phi_i)\subset E_i$, $0 \leq \phi_i \leq 1$, and $\phi_i=1$ outside a compact set containing $K_i$. Define $v_i = -\Delta \phi_i \in C_c^\infty(K)$.  
By Lemmata~\ref{keylemma_HS}, \ref{keylemma_HNS}, and \ref{keylemma}, one can construct exact and approximate solutions $u_i(\cdot,k)$ as follows. For $M$ of the form \eqref{M1},
\begin{align*}
    (\Delta+k^2)u_i - v_i = 0,
\end{align*}
while for $M$ given by \eqref{M2} or \eqref{M3},
\begin{align*}
    (\Delta+k^2)u_i - v_i 
    = O\!\left(\operatorname{ilg}(k)^q |z|^{-\infty}\right)\chi_K + O(k|z|^{-\infty}).
\end{align*}

Denote by $\Delta_i$ the Laplace--Beltrami operator on each end, i.e. $\Delta_i = \Delta_{\mathbb{R}^{n_i}\times M_i}$. A standard computation (see \cite[Sec.~2.2]{HS} and \cite[Sec.~2.1]{HNS}) yields the resolvent kernel estimate
\begin{align}\label{eq_reso}
    (\Delta_i+k^2)^{-1}(z,z') \le C \begin{cases}
        \big(d(z,z')^{2-N} + d(z,z')^{2-n_i}\big)\, e^{-ckd(z,z')}, & n_i \geq 3, \\[4pt]
        \big(d(z,z')^{2-N} + 1+|\log(kd(z,z'))|\big)\, e^{-ckd(z,z')}, & n_i = 2, \\[4pt]
        |K_0(k|z-z'|)|, & \Delta_i = \Delta_{\mathbb{R}^2},
    \end{cases}
\end{align}
where $K_0$ denotes the modified Bessel function of the second kind of order zero.  

For the gradient estimate, one has
\begin{align}\label{eq_gradient_reso}
    |\nabla_z (\Delta_i+k^2)^{-1}(z,z')| \le C \begin{cases}
        \big(d(z,z')^{1-N} + d(z,z')^{1-n_i}\big)\, e^{-ckd(z,z')}, & n_i \geq 3, \\[4pt]
        \big(d(z,z')^{1-N} + d(z,z')^{-1}\big)\, e^{-ckd(z,z')}, & n_i = 2, \\[4pt]
        |z-z'|^{-1} e^{-ck|z-z'|}, & \Delta_i = \Delta_{\mathbb{R}^2}.
    \end{cases}
\end{align}
Finally, following the approach of \cite{HS,HNS}, one obtains for second derivatives
\begin{align}\label{dd}
    |\nabla_z \nabla_{z'} (\Delta_i+k^2)^{-1}(z,z')| \le C \begin{cases}
        \big(d(z,z')^{-N} + d(z,z')^{-n_i}\big)\, e^{-ckd(z,z')}, & n_i \geq 3, \\[4pt]
        \big(d(z,z')^{-N} + d(z,z')^{-2}\big)\, e^{-ckd(z,z')}, & n_i = 2, \\[4pt]
        |z-z'|^{-2} e^{-ck|z-z'|}, & \Delta_i = \Delta_{\mathbb{R}^2}.
    \end{cases}
\end{align}


Next, choose an open set $\widetilde{K}$ containing $K$ such that 
\[
\sum_i \phi_i = 1 \quad \text{in a neighbourhood of } \widetilde{K}^c.
\]
Let $G_{\mathrm{int}}(k)$ denote an interior parametrix for the resolvent at energies $0 < k \leq k_0$, whose kernel is compactly supported in a neighbourhood of the diagonal set 
$
K_{\Delta} := \{(z,z) \,;\, z \in K\}.
$
Moreover, for each $i$, the kernel $G_{\mathrm{int}}(k)(z,z')$ agrees with $(\Delta_i + k^2)^{-1}(z,z')$ in a sufficiently small neighbourhood of $K_{\Delta}$.  
In particular, we assume that 
\[
G_{\mathrm{int}}(k)(z,z') = (\Delta_i + k^2)^{-1}(z,z') 
\quad \text{on the support of } \nabla \phi_i(z)\,\phi_i(z').
\] 
Thanks to \cite{HS,HNS}, the resolvent kernel admits a decomposition into four pieces:
\begin{align*}
(\Delta + k^2)^{-1}(z,z') 
    &= \sum_{j=1}^4 G_j(k)(z,z'), 
    \qquad 0 < k \leq k_0,
\end{align*}
where
\begin{align}\label{G1}
G_1(k)(z,z') &= \sum_{i=1}^l (\Delta_{i} +k^2)^{-1}(z,z') \phi_i(z) \phi_i(z'),\\ \label{G2}
G_2(k)(z,z') &= G_{\textrm{int}}(k)  \left(1-\sum_{i=1}^l \phi_i(z)\phi_i(z')\right),\\ \label{G3}
G_3(k)(z,z') &= \sum_{i=1}^l (\Delta_{i} +k^2)^{-1}(z_i^0,z') u_i(z,k)\phi_i(z').
\end{align}
However, for the term $G_4(k)$, its construction for $M$ defined by \eqref{M1} and \eqref{M2}, \eqref{M3} are significantly different. For the case where there is no parabolic end, i.e., \eqref{M1}, the term $G_4(k)$ is defined by Lemma~\ref{keylemma_HS}. Indeed, one computes the error term $E(k)$ defined by $(\Delta+k^2)\sum_{j=1}^3 G_j(k) = I + E(k)$ explicitly:
\begin{align}\label{eq_error1}
E(k)(z,z') = \sum_{i = 1}^l \Big( -2 \nabla \phi_i(z) \phi_i(z') \big(\nabla_z (\Delta_{i} + k^2)^{-1}(z,z') - \nabla_z G_{\textrm{int}}(z,z')\big)\\ \label{eq_error2}
    +\phi_i(z') v_i(z) \big( -(\Delta_{i} + k^2)^{-1}(z,z') + G_{\textrm{int}}(k)(z,z') + (\Delta_{i} + k^2)^{-1}(z_i^0,z')\big) \Big)\\ \label{eq_error3}
    + \big( (\Delta + k^2)G_{\textrm{int}}(k)(z,z') - \delta_{z'}(z)\big) \left(1-\sum_{i=1}^l \phi_i(z)\phi_i(z')\right).
\end{align}
In the above formula, it is clear to see that $E(k)(z,z')$ is smooth and compactly supported in the $z$-variable. By Lemma~\ref{keylemma_HS}, we define 
\begin{align}
    G_4(k)(z,z') &= - (\Delta+k^2)^{-1} \left(E(k)(\cdot, z')\right)(z)\\ \nonumber
    &= - \int (\Delta + k^2)^{-1}(z,s) E(k)(s,z') ds.
\end{align}
While for the case \eqref{M2} and \eqref{M3}, the situations are different. In fact, the error terms for cases \eqref{M2} and \eqref{M3} are given by the formula:
\begin{align}\label{eq_error4}
    \Tilde{E}&(k)(z,z')\\ \nonumber
    &= E(k)(z,z') + \sum_{i=1}^2 \phi_i(z') \left[ (\Delta+k^2)u_i(z,k) - v_i(z) \right] (\Delta_i+k^2)^{-1}(z_i^0,z') + (\Delta+k^2)\sum_{j=1}^{\mathcal{N}} \psi_i(z)  \varphi_i(z')\\ \nonumber
    &:= \Tilde{E}_1(k)(z,z') + (\Delta+k^2)\sum_{j=1}^{\mathcal{N}} \psi_i(z)  \varphi_i(z'),
\end{align}
where $\mathcal{N}$ is a finite integer and $\{\varphi_i\}$ are a basis of the null space of $I+\Tilde{E}_1(0)$ and $\{\Delta \psi_i\}$ are a basis for a subspace supplementary to the range of $I+\Tilde{E}_1(0)$; see Appendix~\ref{appendix} and \cite[Lemma~3.3, 3.4]{HNS} for details. In particular, both $\varphi_i$ and $\psi_i$ are smooth functions with compact supports. By the invertibility of the operator $I+\Tilde{E}(0)$ \cite[Lemma~3.5]{HNS} (also for $k\le k_0$ for $k_0$ small enough by continuity; see \cite[Lemma~3.1]{HNS}), we invert $I+\Tilde{E}(k)$ and write 
\begin{align*}
    (I+\Tilde{E}(k))^{-1} = I + S(k).
\end{align*}
Note that by construction, we obtain the exact formula for $(\Delta+k^2)^{-1}$ in the sense
\begin{align}\label{parametrix_R2}
    (\Delta+k^2)^{-1} &= \left(\sum_{j=1}^3 G_j(k) + \sum_{j=1}^{\mathcal{N}} \psi_i(\cdot) \langle \varphi_i, \cdot\rangle \right) (I+\Tilde{E}(k))^{-1}\\
    &:= \sum_{j=1}^3 G_j(k) + G_4(k),\quad \forall 0<k\le k_0,
\end{align}
where
\begin{align*}
    G_4(k) = \sum_{j=1}^3 G_j(k)S(k) + \sum_{j=1}^{\mathcal{N}} \psi_i(\cdot) \langle \varphi_i, \cdot\rangle + \sum_{j=1}^{\mathcal{N}} \psi_i(\cdot) \langle \varphi_i, S(k) \cdot \rangle.
\end{align*}

\subsection{Parametrix for the kernel of the Hodge projection}\label{sec2.2}

For simplicity, we write $z_{i}^0 = 0$. Without loss of generality, we assume $\textrm{supp}(\phi_i) \subset \{|z|\ge 1\}$ and $\textrm{supp}(d\phi_i) \subset \{E_i; 1\le |z|\le 2\}$. In this subsection, we construct the main part of the kernel of the perturbation of the Hodge projector:
\begin{align*}
    \mathcal{P}_k = d (\Delta + k^2)^{-1} \delta = \sum_{j=1}^4 d G_j(k) \delta,\quad \forall 0<k\le k_0\le 1/2.
\end{align*}
For main part, we mean the terms involving $G_1, G_3$, i.e., $dG_j(k)\delta$ for $j=1,3$. Let $\omega \in C_c^\infty(\Lambda^1(M))$ and $0<k\le k_0\le 1/2$. Define
\begin{align*}
    \mathcal{T}\omega(z) := dG_1(k)\delta \omega(z),\\
    \mathcal{S}\omega(z):= d G_3(k) \delta \omega(z).
\end{align*}
By the explicit formula \eqref{G1}, it suffices to consider each term in the summation. Note that 
\begin{align*}
    G_1(k) \delta\omega(z) =  \int G_1(k)(z,z') \delta \omega(z') dz'= \phi_i(z) \int (\Delta_i +k^2)^{-1}(z,z') \phi_i(z') \delta \omega(z') dz'.
\end{align*}
By integration by parts, one deduces that
\begin{align*}
    G_1(k) \delta\omega(z) &= \phi_i(z) \int d_{z'}(\Delta_i +k^2)^{-1}(z,z') \phi_i(z') \cdot \omega(z') dz'\\
    &+ \phi_i(z) \int (\Delta_i +k^2)^{-1}(z,z') d\phi_i(z') \cdot \omega(z') dz'\\
    &= \phi_i(z) \int (\Delta_i +k^2)^{-1}(z,z') \delta(\phi_i \omega)(z') dz'\\
    &+ \phi_i(z) \int (\Delta_i +k^2)^{-1}(z,z') d\phi_i(z') \cdot \omega(z') dz'.
\end{align*}
Indeed, the boundary term according to the integration over $\partial B(0,R)$ vanishes as $R\to \infty$ since $\omega$ has compact support. While for the "inner" boundary term (integration over $\partial B(z,\epsilon)$ as $\epsilon \to 0$) is bounded by (here, we only check the case where the $i$th end is parabolic at infinity, and the non-parabolic case is similar)
\begin{align*}
    C_\omega \int_{\partial B(z,\epsilon)} \left|1 + |\log{kd(z,z')}| \right| d\sigma(z') \le C_\omega \epsilon |\log{k\epsilon}| \to 0,
\end{align*}
as $\epsilon \to 0$ for each $0<k\le k_0$.

Therefore, 
\begin{align*}
    \mathcal{T}\omega(z) &= d\phi_i(z)\int d_{z'}(\Delta_i +k^2)^{-1}(z,z') \phi_i(z')\cdot \omega(z') dz' + \phi_i(z) \int d_{z}(\Delta_i +k^2)^{-1}(z,z') \delta(\phi_i \omega)(z') dz'\\
    &+ d\phi_i(z) \int (\Delta_i +k^2)^{-1}(z,z') d\phi_i(z') \cdot \omega(z') dz' + \phi_i(z) \int d_z(\Delta_i +k^2)^{-1}(z,z') d\phi_i(z') \cdot \omega(z') dz'\\
    &:= \mathcal{T}_1 \omega(z)+\mathcal{T}_2 \omega(z) + \mathcal{T}_3 \omega(z) + \mathcal{T}_4 \omega(z).
\end{align*}
Similarly, 
\begin{align*}
    &\mathcal{S}\omega(z) = \int du_i(z,k) (\Delta_i+k^2)^{-1}(0,z') \phi_i(z') \delta \omega(z') dz'\\
    &= du_i(z,k) \int d_{z'}(\Delta_i+k^2)^{-1}(0,z')  \phi_i(z') \cdot \omega(z') dz' + du_i(z,k) \int (\Delta_i+k^2)^{-1}(0,z')  d\phi_i(z') \cdot \omega(z') dz'\\
    &:= \mathcal{S}_1 \omega(z) + \mathcal{S}_2 \omega(z).
\end{align*}
Combing $\mathcal{T}$ and $\mathcal{S}$, one yields
\begin{align*}
    \mathcal{T}\omega + \mathcal{S}\omega &= \mathcal{T}_2\omega + \mathcal{T}_4\omega + (\mathcal{T}_1 + \mathcal{S}_1)\omega + (\mathcal{T}_3+\mathcal{S}_2)\omega\\
    &= \mathcal{T}_2\omega + \mathcal{T}_4\omega + \mathcal{R}_1 \omega + \mathcal{R}_2\omega + \mathcal{Q}_1\omega + \mathcal{Q}_2\omega,
\end{align*}
where
\begin{align*}
    \mathcal{R}_1\omega &= d\phi_i(z) \int  d_{z'}\left[(\Delta_i + k^2)^{-1}(z,z') - (\Delta_i + k^2)^{-1}(0,z')\right] \phi_i(z') \cdot \omega(z') dz',\\
    \mathcal{R}_2\omega &=\left[ d\phi_i(z) + du_i(z,k) \right] \int d_{z'}(\Delta_i + k^2)^{-1}(0,z') \phi_i(z') \cdot \omega(z') dz',\\
    \mathcal{Q}_1\omega &= d\phi_i(z) \int \left[ (\Delta_i + k^2)^{-1}(z,z') - (\Delta_i + k^2)^{-1}(0,z') \right] d\phi_i(z') \cdot \omega(z')dz',\\
    \mathcal{Q}_2\omega &= \left[ d\phi_i(z) + du_i(z,k)\right] \int (\Delta_i + k^2)^{-1}(0,z') d\phi_i(z')\cdot \omega(z') dz'.
\end{align*}

\section{On the crucial part of the parametrix construction }\label{sec3}

We begin by considering the operators $\mathcal{R}_2$ and $\mathcal{Q}_2$. The contributions of $\mathcal{R}_1$ and $\mathcal{Q}_1$ will be discussed later. In the proof, we will make use of the following lemma.

\begin{lemma}\label{lemma_1}
Let $M$ be a manifold defined by \eqref{M1}. Let $1<p<\infty$ and $\omega \in C_c^\infty(\Lambda^1(M))$. We have
\begin{align}
    \lim_{k\to 0}\left( |\langle \mathcal{R}_2\omega, (I-\mathcal{H})\nu \rangle| 
    +|\langle \mathcal{Q}_2 \omega, (I-\mathcal{H})\nu \rangle|
    \right)
    = 0,
\end{align}
where $\nu \in C_c^\infty(\Lambda^1(M)) \subset L^2\cap L^{p'}(\Lambda^1(M))$.

Let $M$ be a manifold defined by \eqref{M2} or \eqref{M3}. Let $1<p<\infty$ and $\omega \in C_c^\infty(\Lambda^1(M))$. We have
\begin{align}
    \lim_{k\to 0} |\langle \mathcal{R}_2\omega, (I-\mathcal{H})\nu \rangle|
    = 0,
\end{align}
where $\nu \in C_c^\infty(\Lambda^1(M)) \subset L^2\cap L^{p'}(\Lambda^1(M))$.
 
\end{lemma}

\begin{proof}

In what follows, we apply Lebesgue dominated convergence theorem. We first consider the case $n_*\ge 3$. In this case, both $|(\Delta_i+k^2)^{-1}(0,z')|$ and $|d_{z'}(\Delta_i+k^2)^{-1}(0,z')|$ converge uniformly as $k\to 0$, and we have equality:
\begin{align*}
    (\Delta+k^2)u_i = -\Delta \phi_i.
\end{align*}
It suffices to show the following bilinear forms:
\begin{align}\label{A}
    &\mathcal{A} =\langle \mathcal{R}_2\omega, (I-\mathcal{H})\nu \rangle \\ \nonumber
    &=\int \left[ d\phi_i(z) + du_i(z,k) \right] \int d_{z'}(\Delta_i+k^2)^{-1}(0,z') \phi_i(z') \cdot \omega(z') dz' \cdot (I-\mathcal{H})\nu(z) dz\\ \nonumber
    &= \left(\int \left[ d\phi_i(z) + du_i(z,k) \right] \cdot (I-\mathcal{H})\nu(z) dz \right) \left(\int d_{z'}(\Delta_i+k^2)^{-1}(0,z') \phi_i(z') \cdot \omega(z') dz' \right),
\end{align}
and
\begin{align}\label{B}
    &\mathcal{B} = \langle \mathcal{Q}_2 \omega, (I-\mathcal{H})\nu \rangle \\ \nonumber
    &= \int \left[ d\phi_i(z) + du_i(z,k) \right] \int (\Delta_i+k^2)^{-1}(0,z') d\phi_i(z') \cdot \omega(z') dz' \cdot (I-\mathcal{H})\nu(z) dz\\ \nonumber
    &= \left(\int \left[ d\phi_i(z) + du_i(z,k) \right] \cdot (I-\mathcal{H})\nu(z) dz \right) \left(\int (\Delta_i+k^2)^{-1}(0,z') d\phi_i(z') \cdot \omega(z') dz' \right),
\end{align}
both tend to zero as $k\to 0$. 

We first treat the integral $\mathcal{I} = \int \left[ d\phi_i(z) + du_i(z,k) \right] \cdot (I-\mathcal{H})\nu(z) dz$. By $L^2$-Hodge decomposition; see \cite{KK},
\begin{align*}
    L^2(\Lambda^1(M)) = \overline{d C_c^\infty(M)}^{L^2} \oplus \overline{\delta C_c^\infty(\Lambda^2(M))}^{L^2} \oplus \mathcal{H}_{L^2}^1(M),
\end{align*}
where $\mathcal{H}_{L^2}^1(M)$ refers to the space of $L^2$ harmonic 1-form on $M$. Without loss of generality, we assume that $(I-\mathcal{H})\nu = d\xi + \delta \eta$ for some $\xi \in C_c^\infty(M)$, $\eta \in C_c^\infty(\Lambda^2(M))$; a density argument will be given in Remark~\ref{remark_density} below. Note that this argument  works for  $\mathcal{B}$ only in the setting of \eqref{M1}. Plugging this decomposition into the integral, we obtain that 
\begin{align*}
\mathcal{I} = \int \left[ d\phi_i(z) + du_i(z,k) \right] \cdot d\xi(z) \, dz 
    + \int d\left[ \phi_i(z) + u_i(z,k) \right] \cdot \delta\eta(z) \, dz
    \\= \int \left[ \Delta\phi_i(z) + \Delta u_i(z,k) \right] \xi(z) \, dz 
    = \int -k^2 u_i(z,k) \, \xi(z) \, dz.
\end{align*}
Meanwhile for the second multiplicand of $\mathcal{A}$, one notes
\begin{align}\label{eq_Lemma_1}
    \left| \int d_{z'}(\Delta_i+k^2)^{-1}(0,z') \phi_i(z') \cdot \omega(z') dz' \right| &\le C \int_{|z'|\ge 1} |z'|^{1-n_i}{|\omega(z')|} dz' \le C \|\omega\|_1 <\infty
\end{align}
since $\omega \in C_c^\infty(\Lambda^1(M))$. Similarly, the second multiplicand of $\mathcal{B}$ can be estimated by
\begin{align}\label{Holder2}
    \left|\int (\Delta_i+k^2)^{-1}(0,z') d\phi_i(z') \cdot \omega(z') dz' \right| &\le C \int_{1\le |z'|\le 2} |z'|^{2-n_i} |\omega(z')| dz'\le C \|\omega\|_1<\infty.
\end{align}
Since $|u_i(z,k)|$ is uniformly bounded, the integrands of $\mathcal{A}$ and $\mathcal{B}$ are pointwise bounded by
\begin{align*}
    C k^2 \|\omega\|_1  |\xi(z)| \to 0
\end{align*}
as $k \to 0$ a.e. Moreover, for $k\le k_0$
\begin{align*}
    |\mathcal{A}|+|\mathcal{B}| \le C_\omega \int |\xi(z)| dz < \infty 
\end{align*}
since $\xi \in C_c^\infty(M)$. By Lebesgue dominated convergence theorem, we conclude that
\begin{align*}
    \mathcal{A} +\mathcal{B} \to 0, \quad \textrm{as $k\to 0$ }.
\end{align*}
This verifies Lemma \ref{lemma_1} for the case  $n_*\ge 3$ i.e. $M$ defined by \eqref{M1}.

Next, we consider $\mathcal{R}_2$ acting on $M=\mathbb{R}^2 \# \mathbb{R}^2$ or $M=(\mathbb{R}^2 \times M_1) \# (\mathbb{R}^{n_2}\times M_2)$. That is we show that $\mathcal{A}\to 0$ as $k\to 0$. The proof is similar to the case $n_*\ge 3$. First, $d_{z'}(\Delta_i+k^2)^{-1}(0,z')$ still converges as $k\to 0$. Second, $u_i(z,k)$ is an approximate solution in the sense:
\begin{align*}
    (\Delta+k^2)u_i + \Delta \phi_i = O\left(\textrm{ilg}(k)^q |z|^{-\infty}\right), \quad q\ge 2.
\end{align*}
Therefore, 
\begin{align*}
    \int \left[ d\phi_i(z) + du_i(z,k) \right] \cdot (I-\mathcal{H})\nu(z) dz = \int \left[ -k^2 u_i(z,k) + O\left(\textrm{ilg}(k)^q |z|^{-\infty}\right) \right]  \xi(z) dz.
\end{align*}
Now, \eqref{eq_Lemma_1} still holds since $|d_{z'}(\Delta_1+k^2)^{-1}(0,z')|$ converges. Thus, the integrand of $\mathcal{A}$ is pointwise bounded by
\begin{align*}
    &C \|\omega\|_1 \left(k^2|u_i(z,k)| + \textrm{ilg}(k)^q |z|^{-\infty} \right) |\xi(z)| \to 0
\end{align*}
as $k \to 0$, since again $u_i(z,k)$ is uniformly bounded. Moreover, for $k\le k_0$
\begin{align*}
    |\mathcal{A}| \le C_\omega \int (1+|z|^{-\infty}) |\xi(z)| dz < \infty,
\end{align*}
completing the proof.

\end{proof}

\begin{remark}\label{remark_density}
In this remark, we give a supplementary density argument for the proof of Lemma~\ref{lemma_1}. Recall by $L^2$-Hodge decomposition,
\begin{align*}
    L^2(\Lambda^1(M)) = \overline{d C_c^\infty(M)}^{L^2} \oplus \overline{\delta C_c^\infty(\Lambda^2(M))}^{L^2} \oplus \mathcal{H}_{L^2}^1(M).
\end{align*}
Let $\nu\in C_c^\infty(\Lambda^1(M))$ and $\epsilon >0$. There exist $\xi \in C_c^\infty(M)$, $\eta \in C_c^\infty(\Lambda^2(M))$ such that $\|d\xi - \mathcal{P}\nu\|_2 \le \epsilon /2$ and $\|\delta \eta - \delta \Delta_2^{-1} d\nu\|_2 \le \epsilon /2$. Therefore,
\begin{align*}
|\langle d\phi_i+du_i, (I-\mathcal{H})\nu \rangle| &\le |\langle d\phi_i+du_i, \mathcal{P}\nu \rangle| + |\langle d\phi_i+du_i, \delta \Delta_2^{-1}d\nu \rangle|\\
    &\le |\langle d\phi_i+du_i, d\xi \rangle| + |\langle d\phi_i+du_i, \delta \eta \rangle| + \epsilon \|d\phi_i+du_i\|_2.
\end{align*}
Observe that $\|d\phi_i+du_i\|_2 \le C$ uniformly for all $0<k\le k_0$. Indeed, for the case $n_* \ge 3$, the uniform $L^2$-boundedness of $du_i(\cdot,k)$ is obvious since $|du_i(z,k)|\le C |z|^{1-n_i} \in L^2$ for all $k\le k_0$; see Lemma~\ref{keylemma_HS}. While for the case $M$ defined by \eqref{M2} or \eqref{M3}, the $L^2$ norm on the non-parabolic end can be treated as before and for the parabolic end $E_1$, we have by Lemma~\ref{keylemma_HNS}, and \ref{keylemma} that 
\begin{align*}
    \int_{E_1} |d\phi_i+du_i(z,k)|^2 dz &\le C \|d\phi_i\|_2^2 + C\int_{E_1} \left(|z|^{-2} + \textrm{ilg}(k) |z|^{-1}\right)^2 e^{-ck|z|} dz\\
    &\le C+ C \textrm{ilg}(k)^2 \int_1^\infty r^{-2} e^{-ckr} rdr\\
    &\le C+ C (\textrm{ilg}(k))^2 \int_{k}^1  e^{-cs} \frac{ds}{s} + C (\textrm{ilg}(k))^2 \\
     &\le C+ C \textrm{ilg}(k) \le C,
\end{align*}
for all $0<k\le k_0$. This implies that
\begin{align*}
     |\langle d\phi_i+du_i, (I-\mathcal{H})\nu \rangle| \le |\langle d\phi_i+du_i, d\xi \rangle| + |\langle d\phi_i+du_i, \delta \eta \rangle| + C\epsilon.
\end{align*}
Now, by \eqref{eq_Lemma_1} and \eqref{Holder2}, one obtains estimates that for $M$ defined by \eqref{M1},
\begin{align*}
    \lim_{k\to 0} |\mathcal{A}| + |\mathcal{B}| \le C \epsilon \|\omega\|_1,
\end{align*}
and for $M$ defined by \eqref{M2} or \eqref{M3}
\begin{align*}
    \lim_{k\to 0} |\mathcal{A}| \le C \epsilon \|\omega\|_1.
\end{align*}
The results follow by letting $\epsilon \to 0$.

\end{remark}

\begin{remark}
We treat the analysis of $\mathcal{Q}_2$ on $M$ specified by \eqref{M2}–\eqref{M3} separately (see Lemma~\ref{remedy} below)—not because Lemma~\ref{lemma_1} is inapplicable, but because the density issue in Remark~\ref{remark_density} requires special care. Recall that
\begin{align}\label{remark2}
\langle \mathcal{Q}_2 \omega, (I-\mathcal{H})\nu \rangle
= \left(\int \bigl[d\phi_i(z)+du_i(z,k)\bigr]\cdot (I-\mathcal{H})\nu(z)\,dz\right)
  \left(\int (\Delta_i+k^2)^{-1}(0,z')\, d\phi_i(z')\cdot \omega(z')\,dz' \right).
\end{align}
By Remark~\ref{remark_density}, the first factor satisfies
\[
\Bigl|\langle d\phi_i+du_i, (I-\mathcal{H})\nu \rangle\Bigr|
\le \bigl|\langle d\phi_i+du_i, d\xi \rangle\bigr|
       + \bigl|\langle d\phi_i+du_i, \delta\eta \rangle\bigr| + C\varepsilon
\le C\varepsilon,
\]
uniformly as $k\to 0$, and hence can be made arbitrarily small by the choice in the density approximation. However, $(\Delta_i+k^2)^{-1}(0,z')$ fails to converge as $k\to 0$; the second factor in \eqref{remark2} grows like $O(|\log k|)$ and therefore diverges as $k\to 0$. In particular, the $C\varepsilon$–smallness afforded by the density argument cannot compensate for this logarithmic blow-up. This necessitates a finer analysis of the behaviour of $d\phi_i(z)+du_i(z,k)$ in the small-$k$ regime.

It is also worth noting that, if a strong-type $L^2$ Hodge decomposition holds on $M$ (see \cite{Gromov}), namely
\[
L^2(\Lambda^1(M))
= d W^{1,2}(M)\;\oplus\; \delta W^{1,2}(\Lambda^2(M))\;\oplus\; \mathcal{H}_{L^2}^1(M),
\]
then the proof of Lemma~\ref{lemma_1} simplifies: the density step from Remark~\ref{remark_density} becomes unnecessary.
\end{remark}

Next, we focus on $\mathcal{Q}_2$ acting on $M$ defined by \eqref{M2}, \eqref{M3}. Note that the contribute of the left varibale of $\mathcal{Q}_2(z,z')$ pointwise tends to zero as $k\to 0$. To see this, by Lemma~\ref{keylemma_HNS}, we know that $u_i(z,0)$ is a bounded solution to the equation: $\Delta u = - \Delta\phi_i$. Therefore, it is clear that $\phi_i(z) + u_i(z,0)$ is a bounded global harmonic function on $M$. However, it follows by \cite[Lemma~2.8, 2.10]{HNS} that the only bounded harmonic function on $M$ is constant, and hence the derivative of $\phi_i(z) + u_i(z,k)$ tends to zero as $k\to 0$. Next, for the integral $\int (\Delta_i + k^2)^{-1}(0,z') d\phi_i(z')\cdot \omega(z') dz'$, it is clear that it blows up with order $|\log{k}|$ as $k\to 0$. So, the question turns to: does the decay in $k$ from the left variable of $\mathcal{Q}_2(z,z')$ near zero fast enough to kill the $|\log{k}|$ from the integral? We answer this question in the following lemma.

\begin{lemma}\label{remedy}
Let $M$ be defined by \eqref{M2} or \eqref{M3}. Then for any $\omega, \nu \in C_c^\infty(\Lambda^1(M))$, we have
\begin{align*}
    \left|\int \mathcal{Q}_2\omega(z) \cdot \nu(z) dz \right| \le C \|\omega\|_{p'} \|\nu\|_p
\end{align*}
for $1<p<2$ and all $0<k\le k_0$.

\end{lemma}

\begin{proof}
We only consider $M$ defined by \eqref{M2} since the proof for $M$ defined by \eqref{M3} is similar. First, by Hölder's inequality, it is easy to check that if $E_i$ is parabolic, then
\begin{align}\label{Holder}
    \left|\int (\Delta_i+k^2)^{-1}(0,z') d\phi_i(z') \cdot \omega(z') dz' \right| &\le C \int \left(1 + \left|\log{k|z'|} \right|\right) e^{-ck|z'|} \chi_{1\le |z'|\le 2}(z') |\omega(z')| dz'\\ \nonumber
    &\le C \|\omega\|_{p'} \left(\int_1^2 \left(1 + \left|\log{kr}\right|\right)^{p} e^{-ckr} rdr  \right)^{1/p}\\ \nonumber
    &\le C \|\omega\|_{p'} \left( k^{-2} \int_k^{2k} |\log{s}|^{p} e^{-s}  s ds \right)^{1/p}\\ \nonumber
    &\le C |\log{k}| \|\omega\|_{p'}.
\end{align}
Hence, by \eqref{Holder2}, one deduces
\begin{align*}
    \left|\int (\Delta_i+k^2)^{-1}(0,z') d\phi_i(z') \omega(z') dz'\right| \le C\begin{cases}
        |\log{k}| \|\omega\|_{p'}, & E_i\quad  \textrm{is parabolic},\\
        \|\omega\|_{p'}, & E_i \quad  \textrm{is non-parabolic}.
    \end{cases}
\end{align*}
From Lemma~\ref{keylemma_HNS} and Lemma~\ref{keylemma}, one has 
\begin{align*}
    |d(\phi_i(z) + u_i(z,k))|\le C \textrm{ilg}(k)|z|^{-1} \chi_{E_1}(z) + Ck|z|^{-1}.
\end{align*}
Directly, 
\begin{align*}
    \left|\int_{E_1} \mathcal{Q}_2\omega(z) \cdot \nu(z) dz \right| \le C\int_{E_1} \textrm{ilg}(k)|z|^{-1} |\log{k}| \, \|\omega\|_{p'} |\nu(z)| dz &\le C \|\omega\|_{p'} \|\nu\|_p \left( \int_1^\infty r^{-p'} rdr \right)^{1/p'}\\
    &\le C \|\omega\|_{p'} \|\nu\|_p
\end{align*}
since $1<p<2$, i.e. $p'>2$. The integral $ \int_{E_2\cup K}$ can be treated similarly.

\end{proof}

\begin{remark}
We mention that Lemma~\ref{lemma_1} in a sense explains why $\mathcal{P}$ can be bounded on $L^p$ for $1<p<\infty$ on $\mathbb{R}^2 \# \mathbb{R}^2$. Indeed, if we consider $\mathbb{R}^n \# \mathbb{R}^n$ for $n\ge 3$, the situation would be different. Note that for $n\ge 3$, the space of $L^2$-harmonic 1-form is no longer trivial, i.e. for $\nu \in L^2(\Lambda^1(M))$, we should assume (the density argument, Remark~\ref{remark_density}, is still applicable)
\begin{align*}
    \nu = d\xi + \delta \eta + \gamma,
\end{align*}
where $\gamma \in L^2(\Lambda^1(M))$ and $d\gamma = \delta \gamma = 0$, and this $\gamma$ term does not annihilate $d\phi_i(z) + du_i(z,k)$ when $k\to 0$. In fact, Theorem~\ref{thm0} guarantees that $\gamma = dh$, where $h \in C^\infty(M)$ is the bounded harmonic function on $M$. Moreover, by Lemma~\ref{keylemma_HS}, $h(z) \to 1$ as $z$ tends to the infinity of one end, say $E_1$ (and tend to 0 at the infinity of another end) and $|dh(z)|\le C|z|^{1-n}$. Then we have integration by parts
\begin{align*}
    \int d [\phi_i(z) + u_i(z,k)] \cdot \gamma(z) dz &= \int [\phi_i(z) + u_i(z,k)] \Delta h(z) dz - \lim_{R\to \infty} \int_{\partial B(o,R)} [\phi_i(z) + u_i(z,k)] \frac{\partial}{\partial \nu}h(z) d\sigma\\
    &= \int \Delta[\phi_i(z) + u_i(z,k)] h(z) dz - \lim_{R\to \infty} \int_{\partial B(o,R)} \frac{\partial}{\partial \nu}[\phi_i(z) + u_i(z,k)] h(z) d\sigma.
\end{align*}
There is no doubt that the main terms above both tend to 0 as $k\to 0$. However, the boundary terms do not vanish (for $R$ large):
\begin{align*}
    &\int_{\partial B(o,R) \cap E_i} \left|\phi_i(z) + u_i(z,k)\right| \left|\frac{\partial}{\partial \nu}h(z)\right| d\sigma \sim  \int_{\partial B(o,R) \cap E_i} |z|^{1-n} d\sigma \sim R^{1-n} R^{n-1} = 1,\\
    &\int_{\partial B(o,R)\cap E_1} \left|\frac{\partial}{\partial \nu}[\phi_i(z) + u_i(z,k)]\right| |h(z)| d\sigma \sim  \int_{\partial B(o,R) \cap E_1} |z|^{1-n} d\sigma \sim 1.
\end{align*}
Therefore, we need to directly analyze the $L^p$-boundedness of the following operator:
\begin{align}\label{4}
    \omega \mapsto \left[ d\phi_i(z) + du_i(z,k) \right] \int d_{z'}(\Delta_i + k^2)^{-1}(0,z') \phi_i(z') \cdot \omega(z') dz'.
\end{align}
Note that for all $k$ small
\begin{align*}
    \left|\int d_{z'}(\Delta_i + k^2)^{-1}(0,z') \phi_i(z') \cdot \omega(z') dz'\right| \le C \|\omega\|_p
\end{align*}
only if 
\begin{align*}
    \int_1^\infty \left|d_{z'}(\Delta_i + k^2)^{-1}(0,z') \right|^{p'} r^{n-1} dr \le C \int_1^\infty r^{(1-n)p'+n-1} dr < \infty,
\end{align*}
i.e. $1<p<n$. Meanwhile,
\begin{align*}
    \|d\phi_i(z) + du_i(z,k)\|_p^p \le C \int_1^\infty \left(|z|^{-\infty} + |z|^{1-n} \right)^p r^{n-1} dr \le C \int_1^\infty r^{(1-n)p+n-1} dr <\infty
\end{align*}
only if $n'<p<\infty$. This implies that the rank one operator defined by \eqref{4} is bounded in $L^p$ for $n'<p<n$ uniformly in $k$, and this range is precisely the range of boundedness of $\mathcal{P}$ acting on $\mathbb{R}^n \# \mathbb{R}^n$ for $n\ge 3$; see \cite[Corollary~3.6]{He2}.

\end{remark}

\begin{lemma}\label{lem3.4}
Let $M$ be a manifold defined by \eqref{M1}, \eqref{M2} or \eqref{M3}. The operator
\begin{align*}
    \mathcal{T}_2: \omega \mapsto \phi_i(z) \int d_{z}(\Delta_i + k^2)^{-1}(z, z')\, \delta(\phi_i \omega)(z')\, dz'
\end{align*}
is bounded on $L^p(\Lambda^1(M))$ for all $1 < p < \infty$, uniformly in $k$ for all $k \le k_0$.
\end{lemma}

\begin{proof}
Obviously,
\begin{align*}
    |\mathcal{T}_2\omega(z)| \le C \chi_{E_i}(z) |d(\Delta_i+k^2)^{-1}\delta (\phi_i \omega)(z)|.
\end{align*}
Note that
\begin{align*}
    d(\Delta_i+k^2)^{-1}\delta = d(\Delta_i +k^2)^{-1/2}\left(d(\Delta_i +k^2)^{-1/2}\right)^*,
\end{align*}
where $d(\Delta_i +k^2)^{-1/2}$ is the local Riesz transform on $\mathbb{R}^{n_i}\times M_i$, which is known to be bounded on $L^p$ for all $1<p<\infty$ with operator norm independent of $k$ when $k\le 1$ (see for example \cite{bak87}). Hence, it follows by duality that 
\begin{align*}
    \|\mathcal{T}_2\omega\|_p \le C \left\| d(\Delta_i +k^2)^{-1/2}\left(d(\Delta_i +k^2)^{-1/2}\right)^* (\phi_i\omega)\right\|_{L^p(\mathbb{R}^{n_i}\times M_i)} \le C \|\omega\|_p,
\end{align*}
as desired.
    
\end{proof}

\begin{lemma}\label{lemma_bad1}
Let $M$ be a manifold defined by \eqref{M1}, \eqref{M2} or \eqref{M3}. The operator
\begin{align*}
    \mathcal{T}_4: \omega \mapsto \phi_i(\cdot) \int d_z(\Delta_i+k^2)^{-1}(\cdot, z') d\phi_i(z') \omega(z') dz'
\end{align*}
is bounded on $L^p(\Lambda^1(M))$ for all $n_*'<p<\infty$ and is of weak type $(n_*', n_*')$ uniformly in $0<k\le k_0$.
\end{lemma}

\begin{proof}
Let $p>n_*'$. By asymptotic formula, the operator has pointwise upper bound:
\begin{align*}
    \chi_{E_i}(z) \int_{\textrm{supp}(d\phi_i)} \left[ d(z,z')^{1-N} + d(z,z')^{1-n_i} \right] |\omega d\phi_i(z')| dz'.
\end{align*}
If $d(z,z')\le 4$, then the kernel is bounded by
\begin{align*}
    C\chi_{E_i}(z) d(z,z')^{1-N} \chi_{\textrm{supp}(d\phi_i)}(z') \chi_{d(z,z')\le 4}.
\end{align*}
A simple Shur's test shows that this diagonal part acts as a bounded operator on $L^p$ for all $1\le p\le \infty$. Next, we consider the off-diagonal part, where $d(z,z')\ge 4$. In this case, the operator is bounded by
\begin{align*}
    \chi_{E_i}(z) \int_{d(z,z')\ge 4} \frac{|\omega d\phi_i(z')|}{d(z,z')^{n_i-1}} dz',
\end{align*}
which is a Riesz potential type operator.

Note that for $z'\in \textrm{supp}(d\phi_i)$ and $d(z,z')\ge 4$, one has $|z|\le 2d(z,z')$ and hence the above is bounded by
\begin{align*}
    \chi_{E_i}(z) |z|^{1-n_i} \|\omega d\phi_i\|_1 \le C \chi_{E_i}(z) |z|^{1-n_i} \|\omega\|_p.
\end{align*}
The result follows directly since $\chi_{E_i}(z) |z|^{1-n_i} \in L^p$ for $p>n_*'$ and it is of weak type $(n_*', n_*')$.

\end{proof}

\begin{lemma}\label{lemma_good1}
Let $M$ be a manifold defined by \eqref{M1}, \eqref{M2} or \eqref{M3}. The operator
\begin{align*}
    \mathcal{R}_1: \omega \mapsto d\phi_i(\cdot) \int d_{z'} \left[(\Delta_i+k^2)^{-1}(\cdot, z') - (\Delta_i+k^2)^{-1}(0,z') \right] \phi_i(z') \omega(z') dz'
\end{align*}
is bounded on $L^p(\Lambda^1(M))$ for all $1<p<\infty$ uniformly in $0<k\le k_0$.
\end{lemma}

\begin{proof}

let $\sigma = \sup_{x\in \textrm{supp}(d\phi_i)}d(x,0)$. If $d(0,z')\le 2 \sigma$, we bound the kernel by (note that in this case $d(z,z')\lesssim 1$)
\begin{equation}\label{eq_local}
    \chi_{\textrm{supp}(d\phi_i)}(z)\chi_{\{E_i:d(0,z')\le 2\sigma\}}(z') \left[d(z,z')^{1-N}+d(0,z')^{1-N}\right].
\end{equation}
It is plain that for all $z\in \textrm{supp}(d\phi_i)$
\begin{align*}
    \int \chi_{\{E_i:d(0,z')\le 2\sigma\}}(z')\left[d(z,z')^{1-N}+d(0,z')^{1-N}\right] dz' &\le C \int_{B(z,3\sigma)} d(z,z')^{1-N}dz' +C \int_{B(0,2\sigma)}d(0,z')^{1-N}dz'\\
     &\le C \int_0^{3\sigma} s^{1-N+N-1} ds
    \le C.
\end{align*}
On the other hand, a similar estimate yields that for all $z'\in E_i$,
\begin{gather*}
    \int \chi_{\textrm{supp}(d\phi_i)}(z) \chi_{\{E_i:d(0,z')\le 2\sigma\}}(z') \left[d(z,z')^{1-N}+d(0,z')^{1-N}\right] dz\lesssim 1.
\end{gather*}
By Schur's test, this local part of the kernel acts as a bounded operator on $L^p$ for $p\in [1,\infty]$. While for the case $d(0,z')\ge 2\sigma$, we use mean value theorem and gradient estimate to get 
\begin{gather}\label{eq_remote}
     |\mathcal{R}_1(z,z')| \chi_{E_i;d(0,z')\ge 2\sigma}(z') \le C \chi_{\textrm{supp}(d\phi_i)}(z) \chi_{E_i;d(0,z')\ge 2\sigma}(z') d(z^*,z')^{-n_i} \le C \chi_{\textrm{supp}(d\phi_i)}(z) \chi_{E_i}(z')d(0,z')^{-n_i},
\end{gather}
where $z^*$ is some point seating on the geodesic jointing $z$ and $0$. Now, it is plain that \eqref{eq_remote} acts boundedly on $L^p$ for all $p\in (1,\infty)$, completing the proof.   
 
\end{proof}

\begin{lemma}\label{lemma_bad2}
The operator
\begin{align*}
    \mathcal{Q}_1: \omega \mapsto d\phi_i(\cdot) \int \left[(\Delta_i +k^2)^{-1}(\cdot, z') - (\Delta_i+k^2)^{-1}(0,z') \right] d\phi_i(z') \omega(z') dz'
\end{align*}
is bounded on $L^p(\Lambda^1(M))$ uniformly in $0<k\le k_0$

$(1)$ for all $1\le p\le \infty$  if $M$ is defined by \eqref{M1},

$(2)$ for all $2<p<\infty$ if $M$ is defined by \eqref{M2},

$(3)$ for all $1<p\le \infty$ if $M$ is defined by \eqref{M3}.
\end{lemma}

\begin{proof}

\textit{Case 1. $M$ defined by \eqref{M1}.}
It is clear to see that the kernel $\mathcal{Q}_1(z,z')$ is bounded by
\begin{align*}
    \chi_{1\le |z|\le 2}(z) \left( d(z,z')^{2-N} + 1     \right) \chi_{1\le |z'|\le 2}(z'),
\end{align*}
which is integrable for both $z$ and $z'$ variables. Indeed, it is plain that
\begin{align*}
    \sup_{1\le |z| \le 2} \int_{1\le |z'| \le 2} d(z,z')^{2-N} dz' + \int_{1\le |z'| \le 2} dz' \le \int_{d(z,z')\le 4} d(z,z')^{2-N} dz' +C \le C,
\end{align*}
and the integrability in $z$-variable follows by symmetry. By Shur's test, $\mathcal{Q}_1$ acts as a bounded operator on $L^p$ for all $1\le p\le \infty$.

\textit{Case 2. $M$ defined by \eqref{M2}.}
By mean value theorem, there exists a point $z^*$ seating on the geodesic jointing $0$ and $z$ such that
\begin{align*}
\left|(\Delta_i +k^2)^{-1}(z, z') - (\Delta_i+k^2)^{-1}(0,z') \right| \le \left|\nabla (\Delta_i + k^2)^{-1}(z^*,z') \right| |z| \le C|z| |z'-z^*|^{-1}.
\end{align*}
Now, by the compact support property of $d\phi_i$, one concludes that
\begin{align*}
    |\mathcal{Q}_1\omega(z)| &\le \chi_{1\le |z|\le 2}(z) \int_{|z^* - z'|\le 4} |z^* - z'|^{-1} |\omega(z')| dz'\\
    &\le C \chi_{1\le |z|\le 2}(z) \|\omega\|_p \left( \int_0^4 r^{-p'} rdr \right)^{1/p'}\\
    &\le C \chi_{1\le |z|\le 2}(z) \|\omega\|_p,
\end{align*}
for all $p>2$.

\textit{Case 3. $M$ defined by \eqref{M3}.}
For this case, the resolvent has formula:
\begin{align*}
    (\Delta_{\mathbb{R}^2}+k^2)^{-1}(z,z') = CK_0(k|z-z'|),
\end{align*}
where $K_0$ is the modified Bessell function of the second type of order 0.

By \cite{AS}, $K_0(x) = -\log{x} + O(|x|)$ for $|x|$ small. Note that in our case, since both $z,z'$ are supported in some compact annulus, we deduce that for $z,z'\in \textrm{supp}(d\phi_i)$,
\begin{align*}
    |(\Delta_{\mathbb{R}^2}+k^2)^{-1}(z,z') - (\Delta_{\mathbb{R}^2}+k^2)^{-1}(0,z')| &= C|-\log{k|z-z'|} + \log{k|z'|} + O(1)| \\
    &= C| \log{|z'|} - \log{|z-z'|} + O(1)|.
\end{align*}
There is no doubt that the kernel $\chi_{1\le |z|\le 2}(z) (|\log{|z'|}|+O(1)) \chi_{1\le |z'|\le 2}(z')$ acts as a bounded operator on $L^p$ for all $1\le p\le \infty$. To this end, one notices that
\begin{align*}
    \chi_{1\le |z|\le 2}(z) \int_{1\le |z'|\le 2} \log{|z-z'|} |\omega(z')| dz' \le C \chi_{1\le |z|\le 2}(z) \|\omega\|_p \left( \int_{0}^4 |\log{r}|^{p'} rdr \right)^{1/p'}\le C \chi_{1\le |z|\le 2}(z) \|\omega\|_p
\end{align*}
for all $1<p\le \infty$. The result follows easily.

\end{proof}

\begin{remark}
It is natural to conjecture that, on the manifold defined by \eqref{M2}, $\mathcal{Q}_1$ should likewise be bounded on $L^p$ for every $1<p<\infty$, with an operator norm independent of sufficiently small $k$. However, Lemma \ref{lemma_bad1} already costs us part of this boundedness range, and the regime $p>2$ is all we need to obtain full-range boundedness of the Hodge projector in \eqref{M2}. We therefore keep Lemma \ref{lemma_bad2} in its present, slightly less accurate form.
\end{remark}

\section{Estimates for the remainder}\label{sec4}

In this section, we analyze the contribution of $G_2$, $G_4$ terms in the perturbation operator $\mathcal{P}_k = d(\Delta+k^2)^{-1} \delta$. 

\subsection{Estimates for $G_2$}

\begin{lemma}\label{lemma_G2}
Let $M$ be a manifold defined by \eqref{M1}, \eqref{M2} or \eqref{M3}. Then
\begin{align*}
    \sup_{0<k\le k_0}\|dG_2(k)\delta\|_{p\to p} \le C,\quad \forall 1< p<  \infty.
\end{align*}
\end{lemma}

\begin{proof}
It is claer that
\begin{align*}
dG_2(k)\delta(z,z') = d_z d_{z'} G_2(k)(z,z'),       
\end{align*}
which is a peudodifferential operator of order zero uniformly in $k\le k_0$ with compact support. So it is bounded on $L^p$ for all $1<p<\infty$.

\end{proof}

\subsection{Estimates for $G_4$}
In this subsection, we consider the term $G_4$. Since the construction of this part is different for $M$ defined by \eqref{M1} and \eqref{M2}, \eqref{M3}, we split our argument into two parts.

\textit{Case 1. For $M$ defined by \eqref{M1}.} Recall that in this case $G_4(k)$ is defined via the key lemma~\ref{keylemma_HS}, appearing in the form
\begin{align}
    G_4(k)(z,z') = - \int (\Delta + k^2)^{-1}(z,s) E(k)(s,z') ds,
\end{align}
where the error term $E(k)$ is given by \eqref{eq_error1}, \eqref{eq_error2} and \eqref{eq_error3}.

Hence,
\begin{align}
    dG_4(k)\delta (z,z') = - \int d_{z}(\Delta + k^2)^{-1}(z,s) d_{z'}E(k)(s,z') ds.
\end{align}
Since $E(k)(z,z')$ is compactly supported in $z$-variable, it follows by Lemma~\ref{keylemma_HS} that
\begin{align}
    |dG_4(k)\delta (z,z')| \le C |z|^{1-n_i} e^{-ck|z|} \| d_{z'}E(k)(\cdot,z')\|_\infty,\quad \forall z\in E_i.
\end{align}
So, the crux is to estimate $|d_{z'}E(k)(z,z')|$. Recall explicit formulas from \eqref{eq_error1}, \eqref{eq_error2} and \eqref{eq_error3}. For \eqref{eq_error1}, applying $d_{z'}$ to each term in the sum, the first line becomes
\begin{align*}
    &-2 \nabla \phi_i(z) \nabla_{z'}\phi_i(z') \big(\nabla_z (\Delta_{i} + k^2)^{-1}(z,z') - \nabla_z G_{\textrm{int}}(z,z')\big)\\
    &-2 \nabla \phi_i(z) \phi_i(z') \big(\nabla_{z'} \nabla_z (\Delta_{i} + k^2)^{-1}(z,z') - \nabla_{z'} \nabla_z G_{\textrm{int}}(z,z') \big).
\end{align*}
Next, by construction, $G_{\textrm{int}}$ coincides with $(\Delta_i+k^2)^{-1}$ near diagonal and on the support of $\nabla \phi_i(z) \phi_i(z')$. Hence, by \eqref{dd}, the above is $O(|z'|^{-n_i} e^{-ck|z'|})$ for all $k\le k_0$.

Similarly, the the second line is
\begin{align*}
    &\nabla_{z'}\phi_i(z') v_i(z) \big( -(\Delta_{i} + k^2)^{-1}(z,z') + G_{\textrm{int}}(k)(z,z') + (\Delta_{i} + k^2)^{-1}(z_i^0,z')\big)\\
    &+ \phi_i(z') v_i(z) \big( - \nabla_{z'}(\Delta_{i} + k^2)^{-1}(z,z') + \nabla_{z'}G_{\textrm{int}}(k)(z,z') + \nabla_{z'}(\Delta_{i} + k^2)^{-1}(z_i^0,z')\big).
\end{align*}
It follows by \eqref{eq_reso}, \eqref{eq_gradient_reso}, the above is again $O(|z'|^{-n_i}e^{-ck|z'|})$. The third line \eqref{eq_error3} is smooth and compactly supported and its estimate is trivial.

In summary, one gets
\begin{align}\label{43}
     \| d_{z'}E(k)(\cdot,z')\|_\infty \le C \begin{cases}
         1, & z'\in K,\\
        |z'|^{-n_i} e^{-ck|z'|}, & z'\in E_i.
     \end{cases}
\end{align}
Thus,
\begin{align}\label{44}
    |dG_4(k)\delta (z,z')| \le C \begin{cases}
        1, & z\in K, z'\in K,\\
        |z'|^{-n_j} e^{-ck|z'|}, & z\in K, z'\in E_j,\\
        |z|^{1-n_i} e^{-ck|z|}, & z\in E_i, z'\in K,\\
        |z|^{1-n_i} e^{-ck|z|} |z'|^{-n_j} e^{-ck|z'|}, & z\in E_i, z'\in E_j.
    \end{cases}
\end{align}
It follows that

\begin{lemma}\label{lemma_G4_1}
Let $M$ be a manifold defined by \eqref{M1}. Then
\begin{align}
    \sup_{0<k\le 1}\|dG_4(k)\delta\|_{p\to p} \le C
\end{align} 
for all $n_*'<p<\infty$.

In addition, $d G_4(k) \delta$ is of weak type $(n_*', n_*')$ uniformly in $0<k\le k_0$.
\end{lemma}

\bigskip

\textit{Case 2. For $M$ defined by \eqref{M2} or \eqref{M3}.} Recall that in these cases
\begin{align*}
    G_4(k) = \sum_{j=1}^3 G_j(k)S(k) + \sum_{j=1}^{\mathcal{N}} \psi_i(\cdot) \langle \varphi_i, \cdot\rangle + \sum_{j=1}^{\mathcal{N}} \psi_i(\cdot) \langle \varphi_i, S(k) \cdot \rangle.
\end{align*}
We first consider the finite rank term $\mathcal{O:=}\sum_{j=1}^{\mathcal{N}} \psi_i(\cdot) \langle \varphi_i, \cdot\rangle$.

\begin{lemma}\label{lemma_O}
Let $M$ be a manifold defined by \eqref{M2} or \eqref{M3}. Then
\begin{align*}
    \sup_{0<k\le k_0}\|d\mathcal{O}\delta\|_{p\to p} \le C,\quad \forall 1\le p\le  \infty.
\end{align*}
\end{lemma}

\begin{proof}
Note that $\mathcal{O}$ is a finite rank operator independent of $k$, with $\psi_i,\varphi_i\in C_c^\infty(M)$. It is clear to see that
\begin{align*}
    d\mathcal{O} \delta (z,z') = \sum_{j=1}^{\mathcal{N}} d\psi_j(z) d\varphi_j(z'),
\end{align*}
and the result follows directly.

\end{proof}

Next, we handle the remainder, i.e., 
\begin{align*}
    \sum_{j=1}^3 G_j(k)S(k) + \mathcal{O}S(k).
\end{align*}
Introduce $\Tilde{G}(k) = \sum_{j=1}^3 G_j(k) + \mathcal{O}$. The kernel of its contribution in the Hodge projector is given by $\int d_z \Tilde{G}(k)(z,s) d_{z'}S(k)(s,z') ds$. Obviously, the crux is to get appropriate estimates for $|d_{z'}S(k)(s,z')|$. Note that by the definition of $S(k)$, one has $I+S(k) = (I+\Tilde{E}(k))^{-1}$ which implies that $S(k) = -\Tilde{E}(k) - \Tilde{E}(k)S(k) = -\Tilde{E}(k) - S(k) \Tilde{E}(k)$. By iterating this formula, one deduces
\begin{align}\label{eq_S}
    S(k) = -\Tilde{E}(k) + \Tilde{E}(k)^2 + \Tilde{E}(k) S(k) \Tilde{E}(k).
\end{align}
Define weight functions
\begin{align*}
    &\mu_1(x,k) = \begin{cases}
        1, & x\in K,\\
        e^{-ck|x|}, & x\in E_1,\\
        |x|^{2-n_2}e^{-ck|x|}, & x\in E_2,
    \end{cases}
    \quad &&\mu_2(x,k) = \begin{cases}
        1, & x\in K,\\
        |x|^{-1}e^{-ck|x|}, & x\in E_1,\\
        |x|^{1-n_2}e^{-ck|x|}, & x\in E_2,
    \end{cases}\\
    &\mu_3(x,k) = \begin{cases}
        1, & x\in K,\\
        |x|^{-2}e^{-ck|x|}, & x\in E_1,\\
        |x|^{-n_2}e^{-ck|x|}, & x\in E_2.
    \end{cases} \quad &&\mu_4(x,k) = \begin{cases}
        1, &x\in K,\\
        (1+|\log{k|x|}|) e^{-ck|x|}, &x\in E_1,\\
        |x|^{2-n_2}e^{-ck|x|}, & x\in E_2.
    \end{cases}
\end{align*}
By \eqref{eq_S}, an argument from \cite[Page 22]{HNS} yields that
\begin{align}\label{eq_S(k)}
    |S(k)(z,z')| \le C |z|^{-\infty} \mu_1(z',k).
\end{align}
Next, we estimate \( d_{z'}\tilde{E}(k)(z,z') \). Recall that
\[
\tilde{E}(k) = \tilde{E}_1(k) + (\Delta + k^2)\mathcal{O},
\]
and by \eqref{eq_error4},
\begin{align*}
    \tilde{E}_1(k)(z,z') &= E(k)(z,z') + \sum_{i=1}^2 \phi_i(z') \left[ (\Delta + k^2) u_i(z,k) - v_i(z) \right] (\Delta_i + k^2)^{-1}(z_i^0, z') \\
    &:= E(k)(z,z') + \tilde{E}_2(k)(z,z').
\end{align*}
Therefore,
\begin{align*}
    S(k) = -E(k) - \tilde{E}_2(k) - (\Delta + k^2)\mathcal{O} - S(k)E(k) - S(k)\tilde{E}_2(k) - S(k)(\Delta + k^2)\mathcal{O}.
\end{align*}
Consequently,
\begin{align}\label{G4_decomposition}
    dG_4(k)\delta - d\mathcal{O}\delta = - d\tilde{G}(k)(I + S(k)) \left(E(k) + (\Delta + k^2)\mathcal{O}\right)\delta - d\tilde{G}(k)(I + S(k))\tilde{E}_2(k)\delta.
\end{align}
For the first term, an argument similar to the case \( n_* \ge 3 \) shows that on the manifold \( M \) defined by \eqref{M2},
\[
|d_{z'} E(k)(z,z')| \le C \chi_{|z| \le 2} \mu_3(z',k), \quad \forall\ 0 < k \le k_0.
\]
The term \( (\Delta + k^2)\mathcal{O} \) is trivial since both \( \psi_i \) and \( \varphi_i \) are smooth and compactly supported. Hence, from \eqref{eq_S(k)} we obtain
\[
|d_{z'}(I + S(k)) \left(E(k) + (\Delta + k^2)\mathcal{O}\right)(s,z')| \le C |s|^{-\infty} \mu_3(z',k).
\]
By \cite[Proposition~3.6]{HNS}, we conclude the estimate
\begin{align}\label{eq_G4_1}
    \left|d\tilde{G}(k)(I + S(k)) \left(E(k) + (\Delta + k^2)\mathcal{O}\right)\delta(z,z') \right| \le C \mu_2(z,k) \mu_3(z',k).
\end{align}
For the term \( d\tilde{G}(k)(I + S(k)) \tilde{E}_2(k)\delta \), we claim that for \( \omega, \nu \in C_c^\infty(\Lambda^1(M)) \),
\begin{align}\label{claim}
    \lim_{k \to 0} \left| \left\langle d\tilde{G}(k)(I + S(k)) \tilde{E}_2(k)\delta\omega, \nu \right\rangle \right| = 0.
\end{align}
Indeed, the argument parallels Lemma~\ref{lemma_1}. For small \( k \), the above is bounded by
\begin{align}\label{eq_bili}
    C \iiint \mu_2(z,k) \left| d_{z'}(I + S(k)) \tilde{E}_2(k)(s,z') \right|\, ds\, |\omega(z')|\, dz'\, |\nu(z)|\, dz.
\end{align}
We now focus on \( \left|d_{z'}(I + S(k)) \tilde{E}_2(k)(s,z')\right| \). From the explicit formula, this is bounded by
\begin{align}\label{eq_E2_?}
    &\left|d_{z'} \tilde{E}_2(k)(s,z')\right| + \left|d_{z'} S(k)\tilde{E}_2(k)(s,z')\right| \\ \nonumber
    &\le \left|d_{z'} \tilde{E}_2(k)(s,z')\right| + C |s|^{-\infty} \int \mu_1(t,k) \left|d_{z'} \tilde{E}_2(k)(t,z')\right| dt.
\end{align}
Note that
\begin{align*}
    \left|d_{z'} \tilde{E}_2(k)(s,z')\right| &\le C \sum_{i=1}^2 |d\phi_i(z')| \mathrm{ilg}(k)^q |s|^{-\infty} |(\Delta_i + k^2)^{-1}(z_i^0,z')| \\
    &\quad + C \sum_{i=1}^2 |\phi_i(z')| \mathrm{ilg}(k)^q |s|^{-\infty} |d_{z'}(\Delta_i + k^2)^{-1}(z_i^0,z')| \\
    &\le C \mathrm{ilg}(k)^q |s|^{-\infty} \chi_{1 \le |z'| \le 2}(z') \mu_4(z',k) + \mathrm{ilg}(k)^q |s|^{-\infty} \mu_2(z',k).
\end{align*}
Substituting this into \eqref{eq_E2_?} gives the upper bound:
\begin{align}\label{eq_G4_L1}
    \left|d_{z'}(I + S(k)) \tilde{E}_2(k)(s,z')\right| \le C \mathrm{ilg}(k)^q |s|^{-\infty} \chi_{1 \le |z'| \le 2}(z') \mu_4(z',k) + \mathrm{ilg}(k)^q |s|^{-\infty} \mu_2(z',k).
\end{align}
Thus, \eqref{eq_bili} is then bounded by
\begin{align*}
    \textrm{ilg}(k)^q \int |\nu(z)| \left(\int_{1\le |z'|\le 2} \mu_4(z',k) |\omega(z')| dz' \right) dz + \textrm{ilg}(k)^q \int |\nu(z)| \left(\int_{|z'|\ge 1} \mu_2(z',0) |\omega(z')| dz' \right) dz. 
\end{align*}
By the proof of Lemma~\ref{lemma_1} and \eqref{Holder}, the integrands above are bounded by 
\begin{align*}
    \textrm{ilg}(k)^{q-1} \|\omega\|_1 |\nu(z)| \to 0, \quad \textrm{pointwise as}\quad  k\to 0.
\end{align*}
A suitable dominating function is then $C|\nu(z)|$ since $\nu \in C_c^\infty(\Lambda^1(M))$. The claim follows by Lebesgue dominated convergence theorem.
A similar argument also works for the case $M$ defined by \eqref{M3} and we omit details.

By \eqref{eq_G4_1}, the following result is immediate. 

\begin{lemma}\label{lemma_G4_2}
Let $M$ be defined by \eqref{M2} or \eqref{M3}. Then
\begin{align*}
    \sup_{0<k\le k_0} \left\| d \Tilde{G}(k) (I+S(k)) \left(E(k) + (\Delta+k^2)\mathcal{O}\right) \delta \right\|_{p\to p} \le C,
\end{align*}
for all $2<p<\infty$.
\end{lemma}

\section{$L^2$ Harmonic 1-Forms}

This section is devoted to the proof of Theorem~\ref{thm0}.

\begin{proof}
Let $\eta$ be a harmonic $1$-form on a manifold $M$ of the form
\[
M \;=\; (\mathbb{R}^{n_1} \times M_1) \# \cdots \# (\mathbb{R}^{n_l} \times M_l).
\]
Define $h \colon M \to \mathbb{R}$ by
\[
h(x) \;=\; \int_\Gamma \eta,
\]
where $\Gamma$ is any smooth path from a fixed base point $x_0 \in M$ to $x$. Since $\eta$ is closed ($d\eta=0$), the value of the integral is independent of the path $\Gamma$ provided $M$ is simply connected (i.e.\ $\pi_1(M)=0$).  

As $\eta$ is harmonic, it is smooth and satisfies $\eta \in L^2(M) \cap L^\infty(M)$. Hence,
\[
|h(x)| \;\leq\; C\, d(x,x_0),
\]
for some positive  constant $C$, where $d(x,x_0)$ denotes the geodesic distance in $M$.

By construction, $dh=\eta$. The space of all such functions $h$ with $dh=\eta \in L^2(M)$ is linear of dimension $l$, generated by functions $h_i$, $i=1,\dots,l$, where each $h_i$ tends to $1$ at the $i$-th end and to $0$ at all others. The corresponding space of $1$-forms spanned by $dh_i$ has dimension $l-1$, since $d1=0$ and $\sum_{i=1}^l h_i = 1$.

Moreover, on each end $E_j$ one has the pointwise estimate (see \cite[Lemma~2.7]{HS})
\[
|dh_i(x)| \;\sim\; C\, |x|^{1-n_j},
\]
which implies that $\eta = dh \in L^p(M)$ if and only if $p>n_j'$, where $n_j'$ is the critical Sobolev exponent associated with the $j$-th end.

Consequently, the harmonic projection operator satisfies
\[
\|\mathcal{H}\|_{L^p \to L^p} < \infty
\quad \Longleftrightarrow \quad
n_*' < p < n_*,
\]
where
$n_* := \min_j n_j $
and $n_*' := \max_j n_j'$.

This completes the proof of Theorem~\ref{thm0} for manifolds of type \eqref{M1}.  
The argument for manifolds of type \eqref{M2} is analogous, while the case \eqref{M3} is established in \cite[Proposition~3.3]{CDS}.
\end{proof}

\section{Proof of the main results}\label{sec6}

We now complete the proofs of Theorems~\ref{thm1} and \ref{thm2}.

\begin{proof}[Proof of Theorem~\ref{thm1}]
Let $M$ be defined by \eqref{M1}.  
Take $\omega, \nu \in C_c^\infty(\Lambda^1(M))$ with $\omega \in L^2 \cap L^p(\Lambda^1(M))$ and $\nu \in L^2 \cap L^{p'}(\Lambda^1(M))$, where $1<p<n_*$.  
Note that
\[
\mathcal{P} = \mathcal{P}(I-\mathcal{H}) = (I-\mathcal{H})\mathcal{P}.
\]
By duality, it suffices to analyze the bilinear form
\[
\langle \mathcal{P}\omega, \nu \rangle \;=\; \langle (I-\mathcal{H})\omega, \mathcal{P}\nu \rangle.
\]

Recall that $\mathcal{P}_k = d(\Delta+k^2)^{-1}\delta = \sum_{j=1}^4 d G_j(k)\delta$ for all $0<k\le k_0$. For such $k$ we write
\begin{align*}
    |\langle (I-\mathcal{H}) \omega, \mathcal{P}\nu \rangle|
    &\le |\langle (I-\mathcal{H}) \omega, \mathcal{P}_k\nu \rangle|
       + |\langle (I-\mathcal{H}) \omega, (\mathcal{P}-\mathcal{P}_k)\nu \rangle| \\
    &\le |\langle (I-\mathcal{H}) \omega, \mathcal{P}_k\nu \rangle|
       + \|(I-\mathcal{H})\omega\|_2 \,\|(\mathcal{P}-\mathcal{P}_k)\nu\|_2.
\end{align*}

We claim that $\mathcal{P}_k \to \mathcal{P}$ strongly in $L^2$, i.e.\
\[
\|(\mathcal{P}-\mathcal{P}_k)\nu\|_2 \to 0 \quad \text{as } k\to 0,
\qquad \forall \nu \in L^2(\Lambda^1(M)).
\]
Indeed,
\[
\mathcal{P}-\mathcal{P}_k 
= k^2 d\Delta^{-1}(\Delta+k^2)^{-1}\delta
= k^2 (\Delta_1+k^2)^{-1}\mathcal{P},
\]
where $\Delta_1$ denotes the Hodge Laplacian on $1$-forms. Let $E_{\Delta_1}(\lambda)$ be the spectral resolution of $\Delta_1$. Then
\begin{align}\label{eq_proof}
    \|(\mathcal{P}-\mathcal{P}_k)\nu\|_2^2
    = \int_0^\infty \left(\frac{k^2}{\lambda+k^2}\right)^2
      d\langle E_{\Delta_1}(\lambda)\mathcal{P}\nu, \mathcal{P}\nu \rangle.
\end{align}
For each $\lambda>0$, the integrand converges to $0$ as $k\to0$ and is uniformly bounded by $1$. Moreover, since $\mathcal{P}\nu$ lies in the orthogonal complement of $\mathcal{H}^1_{L^2}(M)$, we have $E_{\Delta_1}(\{0\})\mathcal{P}\nu=0$. Thus the integrand of \eqref{eq_proof} converges to $0$ almost everywhere with respect to the spectral measure. By the dominated convergence theorem,
\[
\|(\mathcal{P}-\mathcal{P}_k)\nu\|_2 \;\to\; 0 \quad \text{as } k\to0,
\]
since
\[
\int_0^\infty d\langle E_{\Delta_1}(\lambda)\mathcal{P}\nu, \mathcal{P}\nu \rangle
= \|\mathcal{P}\nu\|_2^2 \le \|\nu\|_2^2 < \infty.
\]

It remains to examine the limits
\begin{align}\label{1}
    \lim_{k\to0}\langle (I-\mathcal{H})\omega, \mathcal{P}_k\nu \rangle
    &= \lim_{k\to0}\langle (I-\mathcal{H})\omega, d(G_1(k)+G_3(k))\delta\nu \rangle \\ \label{2}
    &\quad + \lim_{k\to0}\langle (I-\mathcal{H})\omega, dG_2(k)\delta\nu \rangle \\ \label{3}
    &\quad + \lim_{k\to0}\langle (I-\mathcal{H})\omega, dG_4(k)\delta\nu \rangle,
\end{align}
provided each limit is well defined.

We begin with \eqref{1}. By Section~\ref{sec2.2}, this decomposes as
\begin{align*}
    \langle (I-\mathcal{H})\omega, \mathcal{T}\nu \rangle
    + \langle (I-\mathcal{H})\omega, \mathcal{S}\nu \rangle
    &= \sum_{j=2,4} \langle (I-\mathcal{H})\omega, \mathcal{T}_j\nu \rangle \\
    &\quad + \sum_{j=1}^2 \langle (I-\mathcal{H})\omega, \mathcal{R}_j\nu \rangle
           + \sum_{j=1}^2 \langle (I-\mathcal{H})\omega, \mathcal{Q}_j\nu \rangle.
\end{align*}
By Lemma~\ref{lemma_1}, the contributions of $\mathcal{R}_2$ and $\mathcal{Q}_2$ vanish as $k\to0$. By Lemmata~\ref{lem3.4}, \ref{lemma_bad1}, \ref{lemma_good1}, and \ref{lemma_bad2}, the remaining terms are bounded by
\[
C\,\|(I-\mathcal{H})\omega\|_p \|\nu\|_{p'}, \qquad 1<p<n_*,
\]
uniformly in $0<k\le k_0$.

For \eqref{2}, Lemma~\ref{lemma_G2} yields the same bound uniformly in $k$ for all $1<p<\infty$. The argument for \eqref{3} is analogous, by Lemma~\ref{lemma_G4_1}.  

In summary,
\[
\lim_{k\to0} \big|\langle (I-\mathcal{H})\omega, \mathcal{P}_k\nu \rangle\big|
\;\le\; C\,\|(I-\mathcal{H})\omega\|_p \|\nu\|_{p'}, \qquad 1<p<n_*.
\]
The endpoint estimate follows by interpolation and the weak-type bounds in Lemmata~\ref{lemma_bad1} and \ref{lemma_G4_1}. This completes the proof.
\end{proof}

\begin{proof}[Proof of Theorem~\ref{thm2}]
Let $M$ be defined by \eqref{M2} or \eqref{M3}. By a similar argument as above, the $G_1, G_2, G_3$ parts are bounded by $C\|(I-\mathcal{H})\omega\|_p \|\nu\|_{p'}$ for $1<p<2$. Next, for $G_4$, one splits it into $\Tilde{G}(k)S(k)+\mathcal{O}$. The $\mathcal{O}$ term can be handled by Lemma~\ref{lemma_O}, and one further decomposes the remainder as indicated in \eqref{G4_decomposition}. Thanks to the claim \eqref{claim}, part of it vanishes as $k\to 0$, and another part is treated in Lemma~\ref{lemma_G4_2}. 

In summary, we have shown that on $M$,
\begin{align*}
    \|\mathcal{P}\omega\|_p \;\le\; C \,\|(I-\mathcal{H})\omega\|_p,
    \qquad \forall\, 1<p<2.
\end{align*}
Since by Theorem~\ref{thm0} one has $\mathcal{H}_{L^2}^1(M)=\{0\}$, it follows that
$(I-\mathcal{H})\omega = \omega$,
and hence the inequality improves to
\begin{align*}
    \|\mathcal{P}\omega\|_p \;\le\; C \,\|\omega\|_p,
    \qquad \forall\, 1<p<2.
\end{align*}
The extension to all $p$ in the admissible range then follows by duality and the self-adjointness of $\mathcal{P}$.

\end{proof}

\section{Further remarks}\label{Sec7}

\subsection{An alternative method: interior harmonic annihilation}

In this subsection, we call back a claim we made at the end of the introduction. Namely, we claim that Theorem~\ref{thm1} can be alternatively proved by a so-called interior harmonic annihilation method. Given the symmetry of the resolvent kernel, i.e., $(\Delta+k^2)^{-1}(z,z') = (\Delta+k^2)^{-1}(z',z)$, we obtain by parametrix construction that
\begin{align*}
    (\Delta+k^2)^{-1}(z,z') = \sum_{j=1}^4 \tilde{G}_j(k)(z,z'),
\end{align*}
where $\tilde{G}_j(k) = G_j(k)$ for $j=1,2$ and 
\begin{align*}
    \tilde{G}_3(k)(z,z') = \sum_{i=1}^l (\Delta_{i} +k^2)^{-1}(z, z_i^0) u_i(z',k)\phi_i(z).
\end{align*}
Moreover, the error term $\mathcal{E}(k)$ defined by $(\Delta+k^2)\sum_{j=1}^3 \tilde{G}_j(k) = I + \mathcal{E}(k)$ equals
\begin{align*}
\mathcal{E}(k)(z,z') = \sum_{i = 1}^l \Big( -2 \nabla \phi_i(z') \phi_i(z) \big(\nabla_{z'} (\Delta_{i} + k^2)^{-1}(z,z') - \nabla_{z'} G_{\textrm{int}}(z,z')\big)\\ 
    +\phi_i(z) v_i(z') \big( -(\Delta_{i} + k^2)^{-1}(z,z') + G_{\textrm{int}}(k)(z,z') + (\Delta_{i} + k^2)^{-1}(z, z_i^0)\big) \Big)\\ 
    + \big( (\Delta + k^2)G_{\textrm{int}}(k)(z,z') - \delta_{z}(z')\big) \left(1-\sum_{i=1}^l \phi_i(z)\phi_i(z')\right).
\end{align*}
It is obvious that $\mathcal{E}(k)$ has compact support in the right variable $z'$. Then by Lemma~\ref{keylemma_HS}, one defines \begin{align*}
    \tilde{G}_4(k)(z,z') = - (\Delta+k^2)^{-1} \left(E(k)(z, \cdot)\right)(z') = - \int (\Delta + k^2)^{-1}(z',s) \mathcal{E}(k)(z,s) ds.
\end{align*} 
Next, similar to Section~\ref{sec2.2}, we decompose $d \left(\tilde{G}_1(k)+\tilde{G}_3(k) \right) \delta \omega$ into 
\begin{align*}
    \mathcal{T}_1\omega(z) + \mathcal{T}_2 \omega(z) + \tilde{\mathcal{R}}_1 \omega(z) + \tilde{\mathcal{R}}_2 \omega(z) + \tilde{\mathcal{Q}}_1 \omega(z) + \tilde{\mathcal{Q}}_2 \omega(z),
\end{align*}
where $\mathcal{T}_1, \mathcal{T}_2$ are defined as before:
\begin{align*}
    \mathcal{T}_1 \omega(z) &= d\phi_i(z)\int d_{z'}(\Delta_i +k^2)^{-1}(z,z') \phi_i(z')\cdot \omega(z') dz',\\
    \mathcal{T}_2 \omega(z) &= \phi_i(z) \int d_{z}(\Delta_i +k^2)^{-1}(z,z') \delta(\phi_i \omega)(z') dz',
\end{align*}
and
\begin{align*}
    \tilde{\mathcal{R}}_1\omega &= \phi_i(z) \int  d_{z}\left[(\Delta_i + k^2)^{-1}(z,z') - (\Delta_i + k^2)^{-1}(z,0)\right] d\phi_i(z') \cdot \omega(z') dz',\\
    \tilde{\mathcal{R}}_2\omega &= d(\Delta_i+k^2)^{-1}(z,0) \phi_i(z) \int d_{z'}\left[u_i(z',k) +  \phi_i(z')\right] \cdot \omega(z') dz',\\
    \tilde{\mathcal{Q}}_1\omega &= d\phi_i(z) \int \left[ (\Delta_i + k^2)^{-1}(z,z') - (\Delta_i + k^2)^{-1}(z,0) \right] d\phi_i(z') \cdot \omega(z')dz',\\
    \tilde{\mathcal{Q}}_2\omega &= (\Delta_i+k^2)^{-1}(z,0) d\phi_i(z) \int d_{z'}\left[u_i(z',k) +  \phi_i(z')\right] \cdot \omega(z') dz'.
\end{align*}
Now, by Lemma~\ref{lem3.4} and Lemma~\ref{lemma_bad2} ($\textit{Case}$~1), one confirms that $\mathcal{T}_2$ and $\tilde{Q}_1$ are bounded on $L^p$ for all $1<p<\infty$. Therefore, it suffices to prove

\begin{proposition}\label{prop_alter1}
Under the assumptions of Theorem~\ref{thm1}, the following statements hold:

$(\romannumeral1)$ $\mathcal{T}_1$ is bounded on $L^p$ for $1<p<n_*$, uniformly in $k$ for all $k \le k_0$.

$(\romannumeral2)$ $\tilde{\mathcal{R}}_1$ is bounded on $L^p$ for all $1<p<\infty$, uniformly in $k$ for all $k \le k_0$.

$(\romannumeral3)$ For any $\omega \in C_c^\infty(\Lambda^1(M)) \subset L^2\cap L^{p}(\Lambda^1(M))$ ($1<p<\infty$), we have
\begin{align*}
\lim_{k\to 0} \left|\left( \tilde{\mathcal{R}}_2 + \tilde{\mathcal{Q}}_2 \right)  (I-\mathcal{H})\omega \right| = 0.
\end{align*}

\end{proposition}

\begin{proof}
For $(\romannumeral1)$, the proof is parallel to Lemma~\ref{lemma_bad1}. Indeed, for $d(z,z')\le 4$, the kernel is bounded by
\begin{align*}
    |\mathcal{T}_1(z,z')| \le C \chi_{\textrm{supp}(d\phi_i)}(z) d(z,z')^{1-N} \chi_{E_i}(z') \chi_{d(z,z')\le 4}.
\end{align*}
A routine Shur's test guarantees that the above kernel acts as a bounded operator on $L^p$ for all $1\le p\le \infty$. As for $d(z,z')\ge 8$, one has
\begin{align*}
    |\mathcal{T}_1\omega(z)| \le C \chi_{\textrm{supp}(d\phi_i)}(z) \int_{E_i; d(z,z')\ge 8} \frac{|\omega(z')|}{d(z,z')^{n_i-1}} dz'.
\end{align*}
Observe that $|z|\le 2$. It is easy to see that $d(z,z')\ge 2|z'|/3$. Hence, the above is bounded by
\begin{align*}
    \chi_{|z|\le 2} \int_{E_i} \frac{|\omega(z')|}{|z'|^{n_i-1}} dz'.
\end{align*}
By Hölder's inequality, $\mathcal{T}_1$ is bounded on $L^p$ for $1<p<n_i$. The result follows easily.

For $(\romannumeral2)$, the proof follows by Lemma~\ref{lemma_good1}, with the roles of $z,z'$ switched.

To this end, we apply a similar argument as in Lemma~\ref{lemma_1} to verify $(\romannumeral3)$. Observe that on $M$ defined by \eqref{M1}, 
\begin{align*}
    |d(\Delta_i+k^2)^{-1}(z,0) \phi_i(z)| + |d\phi_i(z)||(\Delta_i+k^2)^{-1}(z,0)| \le C |z|^{1-n_i} \chi_{|z|\ge 1} + |z|^{2-n_i} \chi_{1\le |z|\le 2} \le C
\end{align*}
uniformly in $0<k\le k_0$. It is enough to show 
\begin{align}
    \lim_{k\to 0} \left| \int d_{z'}\left[u_i(z',k) +  \phi_i(z')\right] \cdot (I-\mathcal{H})\omega(z') dz'  \right| = 0
\end{align}
for $\omega \in C_c^\infty(\Lambda^1(M))$.

By Remark~\ref{remark_density}, it suffices to assume that $(I-\mathcal{H})\omega = d\xi + \delta \eta$ with $\xi \in C_c^\infty(M)$ and $\eta \in C_c^\infty(\Lambda^2(M))$. As a consequence of integration by parts, we only need to verify that 
\begin{align}\label{asd}
    \lim_{k\to 0} \left| \int d_{z'}\left[u_i(z',k) +  \phi_i(z')\right] \cdot d\xi(z') dz'  \right| &= \lim_{k\to 0} \left| \int \Delta_{z'}\left[u_i(z',k) +  \phi_i(z')\right]  \xi(z') dz'  \right|\\ \nonumber
    &= \lim_{k\to 0} \left| \int -k^2 u_i(z',k) \xi(z') dz' \right|\\ \nonumber
    &\le \lim_{k\to 0} \int k^2 |u_i(z',k)| |\xi(z')| dz'\\ \nonumber
    &=0.
\end{align}
By Lemma~\ref{keylemma_HS}, $|u_i(z',k)|\le (1+|z'|)^{2-n_j}$ for $z'\in E_j$ uniformly in $0<k\le k_0$. We confirm that the integrand of \eqref{asd} pointwise converges to 0 as $k \to 0$. In addition, since $(1+|z'|)^{2-n_*}|\xi(z')|$ is integrable, we conclude the limit \eqref{asd} is 0, completing the proof.

\end{proof}

\medskip

Next, by Lemma~\ref{lemma_G2}, we are left to verify: 

\begin{proposition}\label{prop_alter2}
Under the assumptions of Theorem~\ref{thm1}, 
\begin{align*}
    \sup_{0<k\le k_0} \| d \tilde{G}_4(k) \delta \|_{p\to p} \le C
\end{align*}
for all $1<p<n_*$.
\end{proposition}

\begin{proof}
By a parallel estimate of \eqref{43}, \eqref{44}, one gets
\begin{align}
     \| d_{z}\mathcal{E}(k)(z, \cdot)\|_\infty \le C \begin{cases}
         1, & z\in K,\\
        |z|^{-n_i} e^{-ck|z|}, & z\in E_i,
     \end{cases}
\end{align}
and
\begin{align}
    |d\tilde{G}_4(k)\delta (z,z')| \le C \begin{cases}
        1, & z'\in K, z\in K,\\
        |z|^{-n_j} e^{-ck|z|}, & z'\in K, z\in E_j,\\
        |z'|^{1-n_i} e^{-ck|z'|}, & z'\in E_i, z\in K,\\
        |z'|^{1-n_i} e^{-ck|z'|} |z|^{-n_j} e^{-ck|z|}, & z'\in E_i, z\in E_j.
    \end{cases}
\end{align}
The result follows by Hölder's inequality.
\end{proof}

Now, we may proceed a similar argument from Section~\ref{sec6} to complete the proof of Theorem~\ref{thm1}.

\begin{proof}[Alternative proof of Theorem~\ref{thm1}]
Let $\omega, \nu \in C_c^\infty(\Lambda^1(M))$ such that $\omega \in L^2 \cap L^p(\Lambda^1(M))$ and $\nu\in L^2\cap L^{p'}(\Lambda^1(M))$, where $1<p<n_*$. Note that
\begin{align*}
    \|\mathcal{P}\omega\|_p = \| \mathcal{P}(I-\mathcal{H})\omega\|_p = \sup_{\nu \in C_c^\infty(\Lambda^1(M)); \|\nu\|_{p'}\le 1} \langle \mathcal{P}(I-\mathcal{H})\omega, \nu \rangle.
\end{align*}
Since $\mathcal{P}_k \to \mathcal{P}$ strongly in $L^2$ as $k \to 0$ (see the argument around \eqref{eq_proof}), it suffices to show
\begin{align*}
    \lim_{k\to 0} \left| \langle \mathcal{P}_k(I-\mathcal{H})\omega, \nu \rangle \right| \le C \|(I-\mathcal{H})\omega\|_p \|\nu\|_{p'}.
\end{align*}
Decompose $\mathcal{P}_k = \sum_{j=1}^2 \mathcal{T}_j + \sum_{j=1}^2 \tilde{\mathcal{R}}_j + \sum_{j=1}^2 \tilde{\mathcal{Q}}_j + \sum_{j=2,4} d \tilde{G}_j(k) \delta$. The result follows by Proposition~\ref{prop_alter1}, Lemma~\ref{lem3.4}, Lemma~\ref{lemma_bad2}, Lemma~\ref{lemma_G2} and Proposition~\ref{prop_alter2}. 

\end{proof}

\subsection{Remark on a previous result}

In this subsection, we explain in detail why \cite[Lemma~0.1]{AC} needs some trivial but significant modification. To begin with, we recall their statement including its original proof.

\begin{lemma}\cite[Lemma~0.1]{AC}\label{le_AC}
Let $M$ be a compete Riemannian manifold. Suppose that the Hodge projector onto exact part, i.e. $\mathcal{P}$, is bounded on $L^p$ for some $1<p<\infty$. That is 
\begin{align*}
    \|d \Delta^{-1} \delta \omega \|_p \le C \|\omega\|_p, \quad \forall \omega \in C_c^\infty(\Lambda^1(M)).
\end{align*}
Moreover, assume that the following reverse Riesz inequality holds:
\begin{align*}
    \|\Delta^{1/2}f\|_p \le C \|\nabla f\|_p, \quad \forall f\in C_c^\infty(M).
\end{align*}
Then, the Riesz transform is bounded on $L^{p'}$.
\end{lemma}

\begin{proof}
Consider operator $\Delta^{-1/2}\delta$, the adjoint of the Riesz transform. Then it follows by assumptions above that
\begin{align}\label{eq_inequalities}
    \|\Delta^{-1/2}\delta \omega\|_p = \|\Delta^{1/2} \Delta^{-1} \delta \omega\|_p \le C \|d \Delta^{-1} \delta \omega\|_p \le C \|\omega\|_p.
\end{align}
The result follows by duality.

\end{proof}

The above argument is clear except the step where we apply reverse Riesz inequality. Indeed, let us consider $M = (\mathbb{R}^2 \times \mathbb{S}^{2}) \# \mathbb{R}^4$. Apparently, $\pi_1(M) = \{0\}$. By Theorem~\ref{thm0}, the space of $L^2$-harmonic 1-form on $M$ is trivial. Consequently, Theorem~\ref{thm1} together with a duality argument as in the proof of Theorem~\ref{thm2} show that $\mathcal{P}$ is bounded on $L^p$ for all $1<p<\infty$. Now, if Lemma~\ref{le_AC} holds on $M$, then \cite[Theorem~1.3]{He2} (the boundedness of reverse inequality) guarantees that the Riesz transform is bounded for all $1<p<\infty$, contradicting the main result of \cite{HNS}.

To better explain why reverse Riesz inequality should not hold in the above example, we need to recall some arguments from \cite{He2}. Let us recall the asymptotic formula \eqref{asy}: on a compactification of $\mathbb{R}^n \# \mathbb{R}^n$ ($n\ge 3$), 
\begin{align*}
    \Delta^{-1/2}(x,y) \sim \sum_{j = n-1}^\infty a_j(x)\,|y|^{-j}, \quad x\in \overline{M},\quad y \to \partial \overline{M},
\end{align*}
where the leading coefficient $a_{n-1}$ is a non-trivial bounded harmonic function. As discussed in the introduction (see also \cite[Introduction]{CCH}), the leading coefficient, $d a_{n-1}(x)$, is the main obstacle for the boundedness of Riesz transform for $p\ge n$. However, for suitable $f,g$, one has by the positivity and self-adjointness of $\Delta$ that  
\begin{align}\label{birrt}
    \langle \Delta^{1/2}f, g \rangle = \langle \nabla f, \nabla \Delta^{-1/2} g \rangle.
\end{align}
Suppose that $g$ is supported near the boundary $\partial \overline{M}$, the above asymptotic expansion tells us the problematic leading term of \eqref{birrt} contributes
\begin{align}\label{asyrrt}
    \left \langle  \nabla f, \nabla a_{n-1}(\cdot) \int |y|^{1-n} g(y) dy      \right \rangle = \left \langle  f, \Delta a_{n-1}(\cdot) \int |y|^{1-n} g(y) dy \right\rangle = 0.
\end{align}
This implies that in the bilinear form \eqref{birrt}, the problematic leading term is annihilated by $\nabla f$, convincing us the natural lower bound (this lower bound comes from the duality property between Riesz and reverse Riesz inequalities. In short, the $L^p$ boundedness of Riesz transform directly implies the $L^{p'}$ boundedness of the reverse Riesz inequality; see \cite{CD} for details) of the boundedness of the reverse inequality (i.e., $p=n'$) can be released.

Now, on the example $M=(\mathbb{R}^2 \times \mathbb{S}^{2}) \# \mathbb{R}^4$, we mention that the main issue obstructing the applicability of reverse Riesz inequality is the order of decay of the function $\Delta^{-1}\delta \omega$ at the infinity of the parabolic end. Indeed, in the above example, for $\omega \in C_c^\infty(\Lambda^1(M))$ (so, $\delta \omega \in C_c^\infty(M)$ only), one considers Poisson equation:
\begin{align*}
    \Delta u = \delta \omega.
\end{align*}
By \cite[Lemma~2.9]{HNS}, there exists a unique bounded smooth solution $u$ to the above equation which satisfies asymptotic property: as $|x|\to \infty$ (here, $|x|:= \sup_{a\in K}d(a,x) \ge 1$),
\begin{align*}
    u(x) \sim \begin{cases}
        c, & x\in E_1 = \mathbb{R}^2 \times \mathbb{S}^2,\\
        |x|^{-2}, & x\in E_2 = \mathbb{R}^{4},
    \end{cases}
\end{align*}
for some potentially non-zero constant $c$. However, one checks that one of the key steps of proving reverse Riesz inequality in \cite[Page 9, above equation (3.6)]{He2} is the vanishing of the boundary term of the following integration by parts, 

\begin{align}\label{boundary}
    \int_M \nabla u(x) \nabla \mathcal{U}(x) dx = \int_M u(x) \Delta \mathcal{U}(x) dx - \lim_{R\to \infty} \int_{\partial B(o,R)} u(x) \frac{\partial}{\partial \nu} \mathcal{U}(x) d\sigma,
\end{align}
where $o\in K$ is some fixed point, $\sigma$ is the surface measure and  $\mathcal{U}$ is the unique non-trivial harmonic function on $M$ satisfying asymptotic property:
\begin{align*}
    \mathcal{U}(x) = \begin{cases}
        \log{|x|} + \tilde{c} + O(|x|^{-1}), & \textrm{on the $E_1$ end},\\
        O(|x|^{-2}), & \textrm{on the $E_2$ end}
    \end{cases}
\end{align*}
as $|x|\to \infty$; see \cite[Lemma~2.10]{HNS} and Lemma~\ref{A2} below. Note that on the half sphere $\partial B(o,R) \cap (\mathbb{R}^2 \times \mathbb{S}^2)$, the boundary term of \eqref{boundary} has estimate
\begin{align*}
    \int_{ \partial B(o,R) \cap (\mathbb{R}^2 \times \mathbb{S}^2)} c |x|^{-1} d\sigma \sim R^{-1} R \sim 1,
\end{align*}
as $R\to \infty$. This observation suggests that the method used in \cite{He2} of proving the reverse Riesz inequality on such type of manifolds with ends only works for functions satisfying some decay properties at each end.

Next, we give two remedies for Lemma~\ref{le_AC}. The first one is just by adding the assumptions from the main result of \cite[Theorem~2.1]{AC}, namely the volume doubling condition:
\begin{align}\label{D}
    V(x,2R) \le C V(x,R), \quad \forall x\in M, \quad \forall R>0,
\end{align}
where $V(x,R)$ denotes the volume of the ball centered at $x$ with radius $R$, as well as the $L^p$ scale-invariant Poincaré inequality: for some $1\le p < \infty$ and any ball $B\subset M$,
\begin{align}\label{P_p}
    \int_B |f-f_B|^p d\mu \le C r^p \int_B |\nabla f|^p d\mu,\quad \forall f\in C^\infty(B),
\end{align}
where $r$ denotes the radius of $B$.

\begin{lemma}\label{le_remedy1}
Let $1<p<\infty$. Let $M$ be a complete Riemannian manifold satisfying \eqref{D}, \eqref{P_p}. Suppose the Hodge projector and the reverse Riesz transform are bounded on $L^p$. Then, the Riesz transform is bounded on $L^{p'}$. 
\end{lemma}

\begin{proof}
The key observation is that if one can show that the reverse Riesz ineqiality holds for a broader class of functions in the sense
\begin{align*}
    \|\Delta^{1/2}f\|_p \le C \|\nabla f\|_p
\end{align*}
for all $f\in \dot{W}^{1,p}(M)$, then \eqref{eq_inequalities} holds automatically. In the above setting, $\dot{W}^{1,p}(M)$ refers to the homogeneous Sobolev space with semi-norm $\|f\|_{\dot{W}^{1,p}}:=\|\nabla f\|_p$. Indeed, under the assumptions of \eqref{D} and \eqref{P_p}, as also mentioned by \cite[footnote at page~6]{AC}:" Of course, $f$ can be taken more general than this", the key element of their proof--a Caldéron-Zygmund type decomposition  (\cite[Proposition~1.1]{AC})--works for functions in $\dot{W}^{1,p}$. Then the validity of reverse Riesz inequality follows by \cite[Section~1.2]{AC} and \cite[Section~1.3]{AC} successively.
\end{proof}

\begin{lemma}\label{le_remedy2}
Let $M$ be a manifold with ends in the type of \eqref{M1}. Then the conjunction of boundedness of Hodge projector and reverse Riesz inequality on $L^p$ implies the boundedness of Riesz transform on $L^{p'}$.  
\end{lemma}

\begin{proof}
By the argument above (below Lemma~\ref{le_AC}), it suffices to check that on $M$, reverse Riesz inequality \eqref{rrt} holds for functions in the form: $\Delta^{-1} \delta \omega$ with $\omega \in C_c^\infty(\Lambda^1(M))$. Indeed, since $\delta \omega \in C_c^\infty(M)$, Lemma~\ref{keylemma_HS} guarantees that
\begin{align*}
    |\Delta^{-1}\delta \omega(x)| \le \begin{cases}
        C, & x\in K,\\
        C |x|^{2-n_j}, & z\in E_j.
    \end{cases}
\end{align*}
Note that the above estimates implies 
\begin{align}\label{7.5a}
    \int_M \nabla \Delta^{-1}\delta \omega(x) \nabla \mathcal{U}_i(x) dx = \int_M \Delta^{-1}\delta \omega(x) \Delta \mathcal{U}_i(x) dx = 0.
\end{align}
Indeed the following integration by parts hold: for each $1\le i\le l$,
\begin{align*}
    \int_M \nabla \Delta^{-1}\delta \omega(x) \nabla \mathcal{U}_i(x) dx = \int_M \Delta^{-1}\delta \omega(x) \Delta \mathcal{U}_i(x) dx - \lim_{R\to \infty} \int_{\partial B(o,R)} \Delta^{-1} \delta \omega(x) \frac{\partial}{\partial \nu} \mathcal{U}_i(x) d\sigma = 0,
\end{align*}
where $\mathcal{U}_i$ is the unique bounded harmonic function which tends to 1 at the infinity of $E_i$ and to 0 at the infinity of all other ends. In addition (again, by Lemma~\ref{keylemma_HS}), $|(\partial/\partial \nu) \mathcal{U}_i(x)|\le C |x|^{1-n_j}$ for $x\in E_j$ $(1\le j\le l)$. It follows that 
\begin{align*}
\int_{\partial B(o,R)} |\Delta^{-1} \delta \omega(x)| \left|\frac{\partial}{\partial \nu} \mathcal{U}_i(x) \right| d\sigma \le C \sum_{j} \int_{\partial B(o,R) \cap E_j} |x|^{3-2 n_j} d\sigma \sim C \sum_j R^{2-n_j} \to 0,
\end{align*}
as $R\to \infty$ since $n_j \ge n_* \ge 3$.

The argument from \cite{He2} shows that the equality \eqref{7.5a} implies that the estimates \eqref{eq_inequalities} are valid, see formulas (3.9) and (3.10) from \cite{He2}.
\end{proof}

\appendix

\section{Parametrix for the connected sum of two planes}\label{appendix}

In this part we describe parametrix construction for $M = \mathbb{R}^2 \# \mathbb{R}^2$. We follow a step-by-step procedure as in \cite{HNS}. Since most of the results can be covered by the same arguments in previous works with trivial modifications, we only list the result below and prove the only potential obstacle which we also refer to the key lemma; see Lemma~\ref{keylemma_proof} below.

\begin{lemma}
Let $M$ be defined by \eqref{M3}. On $\partial K \big|_{E_i}$, there exists a unique harmonic extension operator
\begin{align*}
    \mathcal{E}_i : C^\infty(\partial K \big|_{E_i}) \ni f \mapsto u = \mathcal{E}_i f \in C^\infty(E_i)
\end{align*}
such that $u$ tends to a constant at infinity with asymptotic expansion in nonpositive powers of $r:=|z|$ as $r\to \infty$. Moreover, the radial derivative $\partial_r u$ is $O(r^{-2})$ at infinity and has asymptotic expansion obtained by differentiation the expansion for $u$ term by term.
\end{lemma}

\begin{proof}
See \cite[Lemma~2.2]{HNS}.
\end{proof}

\begin{lemma}\label{A2}
Let $F\in L^2(M)$ with compact support in the interior of $K$. Then there exists a unique solution $u$ to the equation $\Delta u = F$ on $M$ such that $u$ is bounded on $E_i$. Moreover, $u$ tends to constants $\beta_i$ as $r\to \infty$ on $E_i$ respectively.

In addition, there exists a unique global harmonic function $\mathcal{U}$ on $M$ such that
\begin{align*}
    \mathcal{U}(z) = \log{r} + c_i + O(r^{-1}), \quad \textrm{as} \quad r\to \infty \quad \textrm{on $E_i$ end}.
\end{align*}
\end{lemma}

\begin{proof}
The result for the first part is followed by \cite[Proposition~2.5, Lemma~2.6, Lemma~2.8]{HNS}. The second part follows by a simple variation of the proof of \cite[Lemma~2.8]{HNS}.
\end{proof}

\begin{lemma}[{\cite[Lemma~2.12]{HNS}}]\label{lemma_HNS_2.12}
Let $u$ be a function on $\mathbb{R}^2$ such that $u \to 0$ as $|x| \to \infty$ and $v = \Delta u$ has compact support. Then $u$ can be extended to $u(x,k)$ for $k \ge 0$ such that $u(x,0)=u(x)$ and 
\[
(\Delta+k^2)u(x,k) - v(x) = O(k |x|^{-\infty}).
\]
Moreover,
\[
|u(x,k)| = O(|x|^{-1} e^{-ck|x|}), 
\qquad 
|\nabla u(x,k)| = O(|x|^{-2} e^{-ck|x|}),
\]
and
\begin{align}\label{deri4}
|\nabla \big(u(x,k) - u(x,0)\big)| = O(k|x|^{-1}).
\end{align}
\end{lemma}

We now turn to the only genuine obstacle in the parametrix construction.

\begin{lemma}\label{keylemma_proof}
Let $v\in C_c^\infty(K)$, and let $\phi$ be the solution to the equation $\Delta \phi = -v$ guaranteed by \cite[Lemma~2.9]{HNS}. Then for any integer $q$, there exists an approximate solution $u(z,k)$ to the equation $(\Delta+k^2)u = v$, in the sense that
\begin{align}\label{eq_error_est}
    (\Delta+k^2)u - v = \chi_K  O(\textrm{ilg}(k)^q |z|^{-\infty}) + O(k |z|^{-\infty}),
\end{align}
and such that for all $0\le k \le k_0$,
\begin{align}\label{eq_u}
    |u(z,k)| \le C \begin{cases}
        1, &z\in K,\\
        e^{-ck|z|}, &z\in E_i,
    \end{cases}
\end{align}
and
\begin{align}\label{eq_du}
    |\nabla u(z,k)|\le C \begin{cases}
        1, &z\in K,\\
        \left(|z|^{-2} + |z|^{-1} \textrm{ilg}(k)\right) e^{-ck|z|}, &z\in E_i.
    \end{cases}
\end{align}
In addition, 
\begin{align}\label{deri3}
|\nabla \left(\phi(z)+u(z,k)\right)| = O(\textrm{ilg}(k)|z|^{-1}).    
\end{align}

\end{lemma}

\begin{proof}
We follow the idea from \cite[Lemma~2.14]{HNS}. Throughout the proof we use notation $r=|z|$ to emphasize that we are using polar coordinate. We set $u(z,0) = -\phi(z)$ and then extend it to $k>0$. By Lemma~\ref{A2}, the function $\phi$ tends to constants $\beta_1$ and $\beta_2$  at the ends $E_1$ and $E_2$ respectively. Thanks to Lemma~\ref{lemma_HNS_2.12}, we may extend function $-\phi+\beta_i$ to $\Phi_i(z,k)$ on $E_i$ for $0<k\le k_0$ such that $\Phi_i(z,0) = -\phi(z)+\beta_i$ and
\begin{align*}
    (\Delta+k^2)\Phi_i(z,k) - v = O(k r^{-\infty}),
\end{align*}
for all $z\in E_i$. Next, we extend $\Phi_i(z,k)$ to $E_i \cup K$ trivially by letting 
\begin{align*}
    \Phi_i(z,k) = - \phi(z) + \beta_i, \quad \forall z\in K.
\end{align*}
Therefore, by construction, $\Phi_1 - \Phi_2 = \beta_1 - \beta_2$ for all $z\in K$. Now, the crux is to construct an approximate solution. Previously, we glued a suitable Bessell function to the parabolic end; see \cite[Lemma~2.14]{HNS}. Since now we have two parabolic ends, we need to patch two Bessell functions to each of them in a weighted way.

Define $\chi_i$ to be a smooth cut-off function such that $\chi_i = 1$ on $E_i$ and $\textrm{supp}(\chi_i) \subset E_i \cup K$ (notice that we choose $K$ large enough so that there exists some space for $\chi_i$ to decay to zero, i.e. for all $z\in \textrm{supp}(\chi_i)$, $\chi_i$ are radial functions depending on $r$ only). Next, we define
\begin{align}\label{eq_u1}
    u_1(z,k) &= -\textrm{ilg}(k) \sum_i \beta_i \chi_i(r) K_0(kr) + \sum_i \chi_i(r) \Phi_i(z,k) \\ \nonumber
    &+ \left(1-\chi_1(r)-\chi_2(r) \right) \left(\theta \Phi_1(z,k) + (1-\theta) \Phi_2(z,k)\right),
\end{align}
where $\theta = \frac{\beta_2}{\beta_2 - \beta_1}$, i.e., $1-\theta = \frac{\beta_1}{\beta_1-\beta_2}$. Note that $u_1(z,k) \to u(z,0) = - \phi(z)$ as $k\to 0$ for each $z$.

By a straightforward computation, one yields
\begin{align*}
    &(\Delta+k^2)u_1(z,k) = -\textrm{ilg}(k) \sum_i \left[\beta_i \Delta \chi_i K_0(kr) - 2 \beta_i \nabla \chi_i \cdot \nabla (K_0(kr))    \right]\\
    &+ \sum_i \Delta \chi_i \Phi_i(z,k) + \sum_i \chi_i (\Delta+k^2) \Phi_i(z,k) - 2 \sum_i \nabla \chi_i \cdot \nabla \Phi_i(z,k)\\
    &+ (1-\chi_1-\chi_2)\left[ \theta (\Delta+k^2)\Phi_1(z,k) + (1-\theta) (\Delta+k^2)\Phi_2(z,k) \right]\\
    &+ \sum_i \Delta \chi_i \left[ -\theta \Phi_1(z,k) - (1-\theta) \Phi_2(z,k) \right] + \sum_i \nabla \chi_i \cdot \left[ 2\theta \nabla \Phi_1(z,k) + 2(1-\theta) \nabla \Phi_2(z,k) \right].
\end{align*}
We first consider the $\Delta \chi_i$ part:
\begin{align*}
    \left[-\textrm{ilg}(k) \beta_1 K_0(kr) + (1-\theta)(\Phi_1-\Phi_2) \right] \Delta \chi_1 + \left[-\textrm{ilg}(k) \beta_2 K_0(kr) + \theta (\Phi_2-\Phi_1) \right] \Delta \chi_2.
\end{align*}
Note that $\Delta \chi_i$ support in $K$, and hence the above equals
\begin{align*}
    \big[\textrm{ilg}(k) \beta_1 &\left( \log(k) + \log{r} + O(1) \right) + (1-\theta)(\beta_1-\beta_2) \big] \Delta \chi_1 \\
    &+ \left[\textrm{ilg}(k) \beta_2 \left( \log(k) + \log{r} + O(1) \right) + \theta (\beta_2-\beta_1) \right] \Delta \chi_2\\
    &= \left[-\beta_1 + O(\textrm{ilg}(k)) \log{r} + \beta_1 \right] \Delta \chi_1 + \left[-\beta_2 + O(\textrm{ilg}(k)) \log{r} + \beta_2 \right] \Delta \chi_2\\
    &= O(\textrm{ilg}(k)) \log{r} \sum_i \Delta \chi_i\\
    &= O(\textrm{ilg}(k) r^{-\infty}) \chi_K,
\end{align*}
where we use the fact $K_0(s) \sim - \log{s} + O(1)$ for $s\le 1$ and $\Phi_1 - \Phi_2 = \beta_1-\beta_2$ for $z\in K$.

Next, we treat the $\nabla \chi_i$ terms:
\begin{align*}
    2 \big[ \beta_1 \textrm{ilg}(k) \nabla(K_0(kr)) + & (1-\theta)\left( \nabla \Phi_2(z,k) - \nabla \Phi_1(z,k) \right) \big] \cdot \nabla \chi_1\\
    &+ 2 \left[ \beta_2 \textrm{ilg}(k) \nabla(K_0(kr)) + \theta \left(\nabla \Phi_1(z,k) - \nabla \Phi_2(z,k) \right) \right] \cdot \nabla \chi_2\\
    &= 2 \textrm{ilg}(k) \nabla (K_0(kr)) \cdot \sum_i \beta_i \nabla \chi_i\\
    &= O(\textrm{ilg}(k) r^{-\infty}) \chi_K,
\end{align*}
since $|\nabla (K_0(kr))|\le C r^{-1} = O(1)$ for $z\in \textrm{supp}(\nabla \chi_i)$. Next, we compare $(\Delta+k^2)u_1$ with $v$. One finds that
\begin{align*}
    (\Delta + k^2)u_1 - v &= \textrm{ilg}(k) v_1 + \sum_i \chi_i \left[ (\Delta+k^2)\Phi_i - v \right]\\
    &+ (1-\chi_1-\chi_2) \left[ \theta \left( (\Delta+k^2)\Phi_1 - v \right) + (1-\theta) \left( (\Delta+k^2)\Phi_2 - v \right) \right]\\
    &= \textrm{ilg}(k) v_1 + O(k r^{-\infty})\chi_{E_1 \cup E_2} + O(k^2 r^{-\infty})\chi_K\\
    &= \textrm{ilg}(k) v_1 + O(k r^{-\infty}),
\end{align*}
where $v_1 \in C_c^\infty(K)$. To complete the proof of \eqref{eq_error_est}, one repeats the above computation, replacing $v$ by $v_1$. Denote by $u_2$ the approximate solution to the equation $(\Delta+k^2)u = v_1$. Then, $u_1 + \textrm{ilg}(k)u_2$ satisfies \eqref{eq_error_est} with $q=2$. Repeating this argument $q$ times, one obtains an approximate solution satisfies the assertion in \eqref{eq_error_est}.

To this end, we check the estimates \eqref{eq_u}, \eqref{eq_du} and \eqref{deri3}. It suffices to check $u_1$ since the other summands for $u$ follow by a similar estimate with an additional factor of some positive power of $\textrm{ilg}(k)$. By \eqref{eq_u1} and Lemma~\ref{lemma_HNS_2.12}, the terms involving $\Phi_i$ satisfy the RHS of \eqref{eq_u}. As for the term $\textrm{ilg}(k) K_0(kr)$, one recalls the asymptotic formula for $K_0$ and concludes that this term is $O(1)$ for $kr$ small and is $O(e^{-ckr})$ for $kr$ large.

Next, for the gradient estimate \eqref{eq_du}, one has by construction that
\begin{align*}
    \nabla u_1(z,k) &= \sum_i \nabla \chi_i \left[ -\textrm{ilg}(k) K_0(kr) \beta_i + \beta_i \right]\\
    &- \textrm{ilg}(k) \nabla(K_0(kr)) \sum_i \beta_i \chi_i\\
    &+ \sum_i \chi_i \nabla \Phi_i + (1-\chi_1-\chi_2) \left[\theta \nabla\Phi_1 + (1-\theta) \nabla \Phi_2\right].
\end{align*}
The estimates for the last line above follow by Lemma~\ref{lemma_HNS_2.12}, which is clear. By asymptotic formula of $K_0$, the first line is $O(1)\chi_K$. As for the second line, it is plain that
\begin{align*}
    \textrm{ilg}(k)|\nabla (K_0(kr))|\le C \begin{cases}
        \textrm{ilg}(k) r^{-1}, & kr\le 1,\\
        r^{-1} e^{-ckr}, & kr\ge 1.
    \end{cases}
\end{align*}
Finally, for \eqref{deri3}, a straightforward calculation yields
\begin{align*}
    \nabla \left(\phi(z) + u_1(z,k)\right) &= \sum_i \nabla \chi_i \left[ -\textrm{ilg}(k) K_0(kr) \beta_i + \beta_i \right]\\
    &- \textrm{ilg}(k) \nabla(K_0(kr)) \sum_i \beta_i \chi_i\\
    &+ \sum_i \chi_i \nabla \left(\Phi_i + \phi \right) + (1-\chi_1-\chi_2) \left[\theta \nabla \left(\Phi_1+\phi \right) + (1-\theta) \nabla \left(\Phi_2+\phi\right)\right].
\end{align*}
By Lemma~\ref{lemma_HNS_2.12} and estimates above, we conclude
\begin{align*}
    |\nabla \left(\phi(z) + u_1(z,k)\right)| = O(\textrm{ilg}(k)r^{-1}) + O(kr^{-1}),
\end{align*}
completing the proof.

\end{proof}

We are now in a position to construct the parametrix for the resolvent at low energy. Namely, we decompose
\begin{align*}
    (\Delta+k^2)^{-1} = \sum_{j=1}^4 G_j(k),\quad \forall 0<k\le k_0
\end{align*}
for some $k_0\le 1/2$ small enough. $G_j(k)$ given as in \eqref{G1}, \eqref{G2}, \eqref{G3}, and \eqref{parametrix_R2}.  
The only difference here is that for each $i=1,2$, $\Delta_i$ is the standard Laplacian on $\mathbb{R}^2$. As in the case of $M$ defined by \eqref{M2}, one must invert $I+\tilde{E}_1(k)$ for small $k$ to obtain the true resolvent.

Define weight
\begin{align*}
    \omega(z) = \begin{cases}
        1, &z\in K,\\
        (r \log{r})^{-1}, &z\in E_i,
    \end{cases}
\end{align*}
and weighted space
\begin{align*}
    L^2_\omega(M) = \{f\in L^2_{\textrm{loc}}(M); \omega^{-1}f \in L^2(M)\}.
\end{align*}
The first step is to show

\begin{lemma}\label{A5}
    The error term $\Tilde{E}_1(k)$ is Hilbert-Schmidt on $L^2_{\omega}(M)$, uniform in $k\le k_0$. Moreover, it is continuous in $k$ and has a limit $\Tilde{E}_1(0)$ as $k\to 0$.
\end{lemma}

\begin{proof}
    See \cite[Lemma~3.1]{HNS}.
\end{proof}

To go further, we need

\begin{lemma}\label{null}
$(1)$ The range of $\Delta \big|_{C_c^\infty(M)}$ is dense in $L^2_{\omega}(M)$.

$(2)$ There exists functions $\{\varphi_i\}_{i=1}^{\mathcal{N}}$ and $\{\psi_i\}_{i=1}^{\mathcal{N}}$ such that $\varphi_i, \psi_i\in C_c^\infty(M)$, and $\{\varphi_i\}$ is a basis of the null space of $I+\tilde{E}_1(0)$ and $\{\Delta \psi_i\}$ is a basis of a subspace supplementary to the range of $I+\tilde{E}_1(0)$.
\end{lemma}

\begin{proof}
See \cite[Lemma~ 3.3, 3.4]{HNS}.
\end{proof}

Now, let $\{\varphi_i\}_{i=1}^{\mathcal{N}}$ and $\{\psi_i\}_{i=1}^{\mathcal{N}}$ be functions satisfying properties in Lemma~\ref{null}. Define finite rank operator
\begin{align*}
    \mathcal{O} = \sum_{i=1}^{\mathcal{N}} \psi_i \langle \varphi_i, \cdot \rangle,
\end{align*}
and 
\begin{align*}
    \tilde{E}(k) = \tilde{E}_1(k) + (\Delta+k^2)\mathcal{O}.
\end{align*}

It follows that

\begin{lemma}
There exists $0<k_0\le 1/2$ such that the operator $I+\tilde{E}(k)$ is invertible for all $0<k\le k_0$. 
\end{lemma}

\begin{proof}
The invertibility of $I+\tilde{E}(0)$ is a consequence of \cite[Lemma~3.5]{HNS}, and the result follows by continuity; Lemma~\ref{A5}.
\end{proof}

Defining $S(k)$ by
\[
(I+\tilde{E}(k))^{-1} = I + S(k),
\]
we finally obtain the true resolvent for $0<k\le k_0$:
\begin{align}\label{parametrix_R2}
(\Delta+k^2)^{-1} 
&= \Bigg(\sum_{j=1}^3 G_j(k) + \mathcal{O}\Bigg)(I+\tilde{E}(k))^{-1} \\
&:= \sum_{j=1}^3 G_j(k) + G_4(k),
\end{align}
where
\[
G_4(k) = \sum_{j=1}^3 G_j(k)S(k) + \mathcal{O} + \mathcal{O}S(k).
\]

\bigskip

\noindent
{\bf Acknowledgements:} The authors would like to thank Pascal Aucher, Gilles Carron, Thierry Coulhon, Baptiste Devyver and Renjin Jiang  for helpful comments. AS was supported by Australian Research Council  Discovery Grant  DP200101065.


\begin{thebibliography}{10}
	
	\bibitem{AS}
	M.~Abramowitz and I.~A. Stegun.
	\newblock {\em Handbook of mathematical functions with formulas, graphs, and
		mathematical tables}, volume~55.
	\newblock US Government printing office, 1948.
	
	\bibitem{ABP}
	M.~Atiyah, R.~Bott, and V.~Patodi.
	\newblock On the heat equation and the index theorem.
	\newblock {\em Inventiones Mathematicae}, 19(4):279--330, 1973.
	
	\bibitem{AC}
	P.~Auscher and T.~Coulhon.
	\newblock Riesz transform on manifolds and {P}oincar\'e{} inequalities.
	\newblock {\em Ann. Sc. Norm. Super. Pisa Cl. Sci. (5)}, 4(3):531--555, 2005.
	
	\bibitem{bak87}
	D.~Bakry.
	\newblock \'{E}tude des transformations de {R}iesz dans les vari\'{e}t\'{e}s
	riemanniennes \`a courbure de {R}icci minor\'{e}e.
	\newblock In {\em S\'{e}minaire de {P}robabilit\'{e}s, {XXI}}, volume 1247 of
	{\em Lecture Notes in Math.}, pages 137--172. Springer, Berlin, 1987.
	
	\bibitem{Carron-noncompact}
	G.~Carron.
	\newblock L2 harmonic forms on non-compact riemannian manifolds.
	\newblock In {\em Proc. Centre Math. Appl. Austral. Nat. Univ}, volume~40,
	pages 49--59, 2002.
	
	\bibitem{Carron1}
	G.~Carron.
	\newblock {$L^2$}-cohomology of manifolds with flat ends.
	\newblock {\em Geom. Funct. Anal.}, 13(2):366--395, 2003.
	
	\bibitem{CCH}
	G.~Carron, T.~Coulhon, and A.~Hassell.
	\newblock Riesz transform and {$L^p$}-cohomology for manifolds with {E}uclidean
	ends.
	\newblock {\em Duke Math. J.}, 133(1):59--93, 2006.
	
	\bibitem{CDS}
	T.~Coulhon, B.~Devyver, and A.~Sikora.
	\newblock Gaussian heat kernel estimates: from functions to forms.
	\newblock {\em J. Reine Angew. Math.}, 761:25--79, 2020.
	
	\bibitem{CoDu}
	T.~Coulhon and X.~T. Duong.
	\newblock Riesz transforms for {$1\leq p\leq 2$}.
	\newblock {\em Trans. Amer. Math. Soc.}, 351(3):1151--1169, 1999.
	
	\bibitem{CD}
	T.~Coulhon and X.~T. Duong.
	\newblock Riesz transform and related inequalities on noncompact {R}iemannian
	manifolds.
	\newblock {\em Comm. Pure Appl. Math.}, 56(12):1728--1751, 2003.
	
	\bibitem{CJKS}
	T.~Coulhon, R.~Jiang, P.~Koskela, and A.~Sikora.
	\newblock Gradient estimates for heat kernels and harmonic functions.
	\newblock {\em J. Funct. Anal.}, 278(8):67, 2020.
	\newblock Id/No 108398.
	
	\bibitem{Davies_nongaussian}
	E.~B. Davies.
	\newblock Non-{G}aussian aspects of heat kernel behaviour.
	\newblock {\em J. London Math. Soc. (2)}, 55(1):105--125, 1997.
	
	\bibitem{DeRham}
	G.~De~Rham.
	\newblock {\em Vari{\'e}t{\'e}s diff{\'e}rentiables: formes, courants, formes
		harmoniques}.
	\newblock FeniXX, 1955.
	
	\bibitem{DY}
	X.~T. Duong and L.~Yan.
	\newblock Duality of {Hardy} and {BMO} spaces associated with operators with
	heat kernel bounds.
	\newblock {\em J. Am. Math. Soc.}, 18(4):943--973, 2005.
	
	\bibitem{GW}
	R.~Greene and H.~Wu.
	\newblock Harmonic forms on noncompact riemannian and k{\"a}hler manifolds.
	\newblock {\em Michigan Mathematical Journal}, 28(1):63--81, 1981.
	
	\bibitem{GIS}
	A.~Grigor'yan, S.~Ishiwata, and L.~Saloff-Coste.
	\newblock Geometric analysis on manifolds with ends.
	\newblock 3:325--343, [2021] \copyright 2021.
	
	\bibitem{GS}
	A.~Grigor'yan and L.~Saloff-Coste.
	\newblock Heat kernel on manifolds with ends.
	\newblock {\em Ann. Inst. Fourier}, 59(5):1917--1997, 2009.
	
	\bibitem{Gromov}
	M.~Gromov.
	\newblock K{\"a}hler hyperbolicity and l2-{Hodge} theory.
	\newblock {\em J. Differ. Geom.}, 33(1):263--292, 1991.
	
	\bibitem{HNS}
	A.~Hassell, D.~Nix, and A.~Sikora.
	\newblock Riesz transforms on a class of non-doubling manifolds {II}.
	\newblock {\em Indiana Univ. Math. J.}, 73(3):955--1003, 2024.
	
	\bibitem{HS1D}
	A.~Hassell and A.~Sikora.
	\newblock Riesz transforms in one dimension.
	\newblock {\em Indiana Univ. Math. J.}, 58(2):823--852, 2009.
	
	\bibitem{HS}
	A.~Hassell and A.~Sikora.
	\newblock Riesz transforms on a class of non-doubling manifolds.
	\newblock {\em Comm. Partial Differential Equations}, 44(11):1072--1099, 2019.
	
	\bibitem{He2}
	D.~He.
	\newblock Reverse riesz inequality on manifolds with ends, 2024.
	
	\bibitem{He1}
	D.~He.
	\newblock Endpoint estimates for {R}iesz transform on manifolds with ends.
	\newblock {\em Ann. Mat. Pura Appl. (4)}, 204(1):245--259, 2025.
	
	\bibitem{Hodge}
	W.~V.~D. Hodge.
	\newblock {\em The {T}heory and {A}pplications of {H}armonic {I}ntegrals}.
	\newblock Cambridge University Press, Cambridge; The Macmillan Company, New
	York, 1941.
	
	\bibitem{KK}
	K.~Kodaira.
	\newblock Harmonic fields in {R}iemannian manifolds (generalized potential
	theory).
	\newblock {\em Ann. of Math. (2)}, 50:587--665, 1949.
	
	\bibitem{Li}
	P.~Li.
	\newblock Harmonic sections of polynomial growth.
	\newblock {\em Math. Res. Lett.}, 4(1):35--44, 1997.
	
	\bibitem{LiWa}
	P.~Li and J.~Wang.
	\newblock Counting dimensions of {{\(L\)}}-harmonic functions.
	\newblock {\em Ann. Math. (2)}, 152(2):645--658, 2000.
	
	\bibitem{Li2}
	X.-D. Li.
	\newblock On the strong {$L^p$}-{H}odge decomposition over complete
	{R}iemannian manifolds.
	\newblock {\em J. Funct. Anal.}, 257(11):3617--3646, 2009.
	
	\bibitem{Li1}
	X.-D. Li.
	\newblock Riesz transforms on forms and {$L^p$}-{H}odge decomposition on
	complete {R}iemannian manifolds.
	\newblock {\em Rev. Mat. Iberoam.}, 26(2):481--528, 2010.
	
	\bibitem{presik}
	M.~Preisner and A.~Sikora.
	\newblock Hardy spaces meet harmonic weights revisited, 2023.
	
	\bibitem{hode}
	G.~Schwarz.
	\newblock {\em Hodge decomposition---a method for solving boundary value
		problems}, volume 1607 of {\em Lecture Notes in Mathematics}.
	\newblock Springer-Verlag, Berlin, 1995.
	
	\bibitem{shen}
	Z.~Shen.
	\newblock Bounds of {Riesz} {Transforms} on $l^p$ {Spaces} for {Second} {Order}
	{Elliptic} {Operators}.
	\newblock {\em Annales de l'Institut Fourier}, 55(1):173--197, 2005.
	
	\bibitem{Simpson}
	C.~T. Simpson.
	\newblock Constructing variations of hodge structure using yang-mills theory
	and applications to uniformization.
	\newblock {\em Journal of the American Mathematical Society}, 1(4):867--918,
	1988.
	
	\bibitem{Str}
	R.~S. Strichartz.
	\newblock Analysis of the {L}aplacian on the complete {R}iemannian manifold.
	\newblock {\em J. Functional Analysis}, 52(1):48--79, 1983.
	
\end{thebibliography}
\end{document}